\ifdefined\AddToHook
  \let\ORIlabel\label
  \let\ORIrefstepcounter\refstepcounter
  \AddToHook{package/hyperref/before}{%
    \let\label\ORIlabel
    \let\refstepcounter\ORIrefstepcounter}
\fi
\documentclass[onefignum,onetabnum]{siamart220329}

\usepackage[T1]{fontenc}
\usepackage[latin1]{inputenc}
\usepackage{caption}
\usepackage{algorithm,algpseudocode}
\usepackage{graphicx,url}
\usepackage{amsmath,amssymb,amsopn}
\usepackage{tabularx}
\usepackage{color}
\usepackage{calc}
\usepackage{longtable}
\usepackage{multirow}
\usepackage{mathrsfs}
\usepackage{array}
\usepackage{hyperref}
\usepackage{cleveref}
\usepackage{tikz}
\usepackage{booktabs}
\usepackage{pgfplotstable}
\usepackage{pgfplots}
\pgfplotsset{compat=1.17}
\usepackage[outline]{contour}
\usepackage{calrsfs}
\usepackage{graphicx}
\usepgfplotslibrary{groupplots}
\usetikzlibrary{matrix}
\usetikzlibrary{automata, shapes.geometric, fit}

\def \dv{\vec{d}}
\def \ev{\vec{e}}

\def \gv{\vec{g}}
\def \hv{\vec{h}}

\def \mv{\vec{m}}

\def \qv{\vec{q}}

\def \sv{\vec{s}}

\def \uv{\vec{u}}
\def \vv{\vec{v}}
\def \wv{\vec{w}}
\def \xv{\vec{x}}
\def \yv{\vec{y}}
\def \zv{\vec{z}}
\def \thetav{\vec{\theta}}

\def \Dv{\vec{D}}
\def \Ev{\vec{E}}

\def \Hv{\vec{H}}

\def \Kv{\vec{K}}

\def \Pv{\vec{P}}
\def \Qv{\vec{Q}}
\def \Rv{\vec{R}}

\def \Vv{\vec{V}}
\def \Wv{\vec{W}}

\def \rm{\mat{r}}

\def \R{\mathbb{R}}			

\let\Lpolish\L 
\def \L{\pazocal{L}}			%
 
\DeclareMathAlphabet{\pazocal}{OMS}{zplm}{m}{n}

\renewcommand{\vec}[1]{\boldsymbol{#1}}
\newcommand{\mat}[1]{\boldsymbol{{#1}}}

\algnewcommand\algorithmiconput{\textbf{Constants:}}
\algnewcommand\algorithmicinput{\textbf{Input:}}
\algnewcommand\algorithmicoutput{\textbf{Output:}}
\algnewcommand{\algorithmicgoto}{\textbf{go to}}%

\algnewcommand\Constants{\item[\algorithmiconput]}
\algnewcommand\Input{\item[\algorithmicinput]}%
\algnewcommand\Output{\item[\algorithmicoutput]}%
\algnewcommand{\Goto}[1]{\algorithmicgoto~\ref{#1}}%

\definecolor{myblack}{RGB}{53, 53, 53}
\definecolor{myblue}{RGB}{40, 75, 99}
\definecolor{myred}{RGB}{192, 50, 33}
\definecolor{myyellow}{RGB}{255, 166, 48}
\definecolor{mywhite}{RGB}{240, 237, 238}
\definecolor{mygreen}{RGB}{0, 102, 0}

\definecolor{green1}{RGB}{9, 82, 86}
\definecolor{green2}{RGB}{8, 127, 140}
\definecolor{green3}{RGB}{6, 167, 125}
\definecolor{green4}{RGB}{79, 109, 122}
\definecolor{green5}{RGB}{192, 214, 223}
\definecolor{violet}{RGB}{26,69,131}

\definecolor{checkgreen}{rgb}{0,0.6,0}
\definecolor{phase1}{rgb}{0.008,0.655,1.000}
\definecolor{phase2}{rgb}{0.016,0.75,0.700}
\definecolor{phase3}{rgb}{0.929,0.35,0.700}
\definecolor{icsyellow}{cmyk}{0.00,0.11,0.53,0.00}

\definecolor{blackmy}{RGB}{38, 70, 83}
\definecolor{bluemy}{RGB}{39, 125, 161}
\definecolor{greenmy}{RGB}{42, 167, 143}
\definecolor{yellowmy}{RGB}{233, 196, 106}
\definecolor{brownmy}{RGB}{244, 162, 97}
\definecolor{redmy}{RGB}{249, 65, 68}

\definecolor{darkbluemy}{RGB}{65, 59, 147}
\definecolor{lightbluemy}{RGB}{71, 139, 194}
\definecolor{greenmy}{RGB}{98, 173, 153}
\definecolor{darkorangemy}{RGB}{230, 142, 52}
\definecolor{lightorangemy}{RGB}{217, 172, 59}

\definecolor{blue1}{RGB}{1, 42, 74}
\definecolor{blue2}{RGB}{1, 73, 124}
\definecolor{blue3}{RGB}{42, 111, 151}
\definecolor{blue4}{RGB}{44, 125, 160}
\definecolor{blue5}{RGB}{70, 143, 175}
\definecolor{blue6}{RGB}{137, 194, 217}

\definecolor{red1}{RGB}{204, 68, 75}
\definecolor{red2}{RGB}{218, 85, 82}
\definecolor{red3}{RGB}{227, 150, 149}
\definecolor{red4}{RGB}{228, 190, 171}

\definecolor{brown1}{RGB}{92,178,112}
\definecolor{brown2}{RGB}{130,194,110}		
\definecolor{brown3}{RGB}{163, 193, 173}

\usepackage[normalem]{ulem}
\usepackage{soul,xcolor}
\usepackage[many]{tcolorbox}
\usepackage{mathtools}
\usepackage{algorithm}
\usepackage{enumitem}
\usepackage{listings}
\usepackage{relsize}
\usepackage{nameref}

\usepackage{longtable}

\newcommand{\algorithmiccommentMine}[1]{\bgroup\hfill$\triangleright$~{ \textcolor{gray}{#1}}\egroup}
\newcommand\COMMENTmine[1]{\algorithmiccommentMine{#1}}

\newcommand\oldtext[1]{}
\newcommand\cancel[1]{}

\newcommand\oldtextt[1]{}

\usetikzlibrary{spy}
\usepackage{pgfplots}
\usepgfplotslibrary{fillbetween}

\usepackage{array}
\newcommand{\PreserveBackslash}[1]{\let\temp=\\#1\let\\=\temp}
\newcolumntype{C}[1]{>{\PreserveBackslash\centering}p{#1}}
\newcolumntype{R}[1]{>{\PreserveBackslash\raggedleft}p{#1}}
\newcolumntype{L}[1]{>{\PreserveBackslash\raggedright}p{#1}}

\DeclareMathAlphabet\bpazocal{OMS}{cmsy}{b}{n}
\usetikzlibrary{patterns}

\usetikzlibrary{arrows.meta,positioning}
\tikzset{every path/.style={draw opacity=1}, every node/.style={text opacity=1}}

\newsiamthm{assumption}{Assumption}
\crefname{assumption}{Assumption}{Assumptions}
\newsiamremark{remark}{Remark}
\newsiamremark{example}{Example}
\crefname{example}{Example}{Examples}

\headers{Safeguarded neurally-accelerated preconditioning}{Y.~Lee, G.~E.~Karniadakis, and A.~Kopani\v{c}\'akov\'a}

\title{Safeguarded Neurally-Accelerated Preconditioning (SNAP) framework for  Nonlinear Optimization\thanks{Submitted to the editors \today.
\funding{This work benefited from ANITI (Grant No.~ANR-23-IACL-0002).
The numerical results were carried out using HPC resources from GENCI-IDRIS (Grant No.~AD011015766R1).}}}

\author{Youngkyu Lee\thanks{Department of Mathematics, Kyung Hee University, Seoul, 02447, Republic of Korea, and Division of Applied Mathematics, Brown University, Providence, RI 02906, USA (\email{youngkyu\_lee@khu.ac.kr}).}
\and George Em Karniadakis\thanks{Division of Applied Mathematics, Brown University, Providence, RI 02906, USA (\email{george\_karniadakis@brown.edu}).}
\and Alena Kopani\v{c}\'akov\'a\thanks{Corresponding author. Department of Computer Science, University of Geneva, Route de Drize 7, 1227 Carouge, Switzerland, and Artificial and Natural Intelligence Toulouse Institute (ANITI), 3 Rue Tarfaya, 31400 Toulouse, France (\email{Alena.Kopanicakova@unige.ch}).}}

\ifpdf
\hypersetup{
  pdftitle={Safeguarded Neurally-Accelerated Preconditioning (SNAP) framework for  Nonlinear Optimization},
  pdfauthor={Y. Lee, G. E. Karniadakis, A. Kopanicakova}
}
\fi

\begin{document}

\maketitle

\begin{abstract}
We introduce the Safeguarded Neurally-Accelerated Preconditioning (SNAP) framework for hybridizing nonlinear optimization methods with pre-trained neural operators (NOs).
SNAP targets parametric optimization problems arising from variational formulations of nonlinear partial differential equations (PDEs), where many parameter instances require repeated solves.
Inspired by multilevel line-search methods, SNAP operates in cycles, combining classical optimization steps with a search direction generated by a pre-trained NO.
Since NO-generated search directions carry no a priori descent guarantee, SNAP employs an a posteriori safeguard by accepting only corrections that satisfy prescribed descent and acceptance conditions.
We prove global convergence of the SNAP algorithm and show that it preserves the convergence rate of the underlying optimization method independently of the employed NO architecture.
Moreover, we also show that when the NO corrections are sufficiently accurate, SNAP converges faster, with the improvement directly related to the NO training error.
To demonstrate SNAP's efficiency in practice, we run numerical experiments on a diverse set of nonlinear parametric optimization problems, including parametric minimal surface, Allen--Cahn, and nonlinear Helmholtz problems.
The efficiency of SNAP is demonstrated by comparing two NO architectures, Fourier neural operator (FNO) and Transolver, and by using SNAP as a standalone solver and as a nonlinear preconditioner.
Our numerical results show a substantial reduction in the number of required iterations, while demonstrating robustness and generalization across a wide range of problem parameters and mesh resolutions.
\end{abstract}



\begin{sloppypar}
\section{Introduction}
\label{sec:numerical_solution}
In this work, we consider nonlinear parameter-dependent partial differential equations (PDEs) with a variational structure, whose discretization gives rise to the family of the following minimization problems:
\begin{align}
\uv^\star(\boldsymbol{\lambda})
\in
\operatorname*{argmin}_{\uv\in\R^n}
\Psi(\uv;\boldsymbol{\lambda}),
\qquad
\boldsymbol{\lambda}\in\Lambda,
\label{eq:dp_merit}
\end{align}
where $\Psi(\cdot; \boldsymbol{\lambda})\colon\R^n\rightarrow\R$ is a twice continuously differentiable, possibly non-convex, discrete energy functional, and $n$ denotes the number of degrees of freedom (dofs).
The vector $\boldsymbol{\lambda}\in\Lambda$ denotes the problem parameters. 
Here, $\Lambda\subset\R^{n_{\lambda}}$, with $n_{\lambda}\geq1$, is a compact set of admissible parameters.
Finally, $\uv^\star(\boldsymbol{\lambda})\in\R^n$ denotes a local minimizer of $\Psi(\cdot;\boldsymbol{\lambda})$ for a given parameter vector $\boldsymbol{\lambda}$.

Problems of this type arise in many scientific and engineering applications, where the numerical solution must be obtained repeatedly for varying problem parameters, e.g., material properties, boundary conditions, or source terms.
A classical approach for solving~\eqref{eq:dp_merit} is to employ first- or second-order optimization methods~\cite{nocedal2006numerical}.
First-order methods have low per-iteration cost but typically converge slowly for ill-conditioned problems.
Second-order Newton methods~\cite{nocedal2006numerical} are well known for local quadratic convergence, but their convergence is not global and it is therefore generally guaranteed only sufficiently close to a solution.
Consequently, they are commonly equipped with globalization strategies, such as line-search (LS) or trust-region (TR) methods~\cite{deuflhard2011newton}, to ensure convergence from remote initial guesses.
Moreover, each Newton iteration requires solving a large-scale linear system of equations.
These systems are commonly solved using Krylov subspace methods~\cite{saad2003iterative}, whose convergence can be accelerated by means of preconditioning.
In particular, multigrid (MG)~\cite{trottenberg2000multigrid} and domain-decomposition (DD) methods~\cite{toselli2005domain} are widely used as they often provide scalability and problem-size independent convergence.

As an alternative to solving the linearized systems arising within Newton's method, one may apply MG or DD techniques directly to the nonlinear problem.
Prominent nonlinear optimization-based MG approaches include MG/OPT~\cite{Nash2000multigrid}, recursive multilevel trust-region (RMTR) methods~\cite{gratton2006numerical,Gratton2008recursive} or multilevel cubic regularization~\cite{calandra2021high}.
Similarly, nonlinear DD-based optimization methods have been developed in the literature.
For example, additively preconditioned LS/TR methods (APLS/APTR) methods~\cite{gross2009unifying}, DD-TR~\cite{gratton2019optimization}, or G-ASPIN~\cite{gross2021globalization}.
As in the linear setting, these nonlinear MG and DD optimization methods can be used either as standalone solvers or as nonlinear preconditioners for nonlinear Krylov or acceleration methods~\cite{brune2015composing}, e.g., nonlinear GMRES or Anderson acceleration.

While nonlinear MG and DD methods can be highly effective, their performance often depends on problem-specific components, such as hierarchical and domain decompositions, smoothers, coarse spaces, transfer operators, and subdomain solvers.
Designing these components however require considerable expertise, and effective choices may not transfer across problem classes or parameter regimes.
This motivates the development of hybrid, machine-learning (ML)-enhanced, optimization methods that use ML models to improve problem-dependent components while retaining the robustness and convergence mechanisms of established optimization techniques.

The use of ML to accelerate iterative methods has recently attracted significant interest in the literature.
For instance, ML has been employed for solver selection~\cite{bhowmick2006application,motter2015lighthouse},
for choosing appropriate solver parameters~\cite{hospedales2021meta,margenberg2022neural,antonietti2021accelerating}, and for providing suitable initial guesses~\cite{eshaghi2025nows,jin2025fast,novello2024accelerating}.
More tightly coupled approaches integrate learned components directly into the design of iterative algorithms. 
This includes for example learned update rules~\cite{hao2024newton,hsieh2019learning,tompson2017accelerating}, solver-in-the-loop strategies~\cite{doncevic2022recursively,um2020solver}, or learned Krylov search directions~\cite{kaneda2022deep,luo2024neural}.
Alternative approaches include data-driven preconditioners based on learned sparsity patterns, approximate inverses, graph representations, or problem-specific structure, see for example~\cite{dimola2025neural,dimola2025numerical,herb2026accelerating,rudikov2024neural,wang2026learnable}.

ML has also been increasingly incorporated into linear DD and MG methods.
For example, in the DD methods, this includes replacing local subdomain solves by learned surrogates~\cite{lee2025nonoverlapping,li2019d3m,li2020deep}, constructing adaptive coarse spaces~\cite{ciaramella2022spectral,heinlein2019machine,klawonn2022learning}, or learning interface transmission conditions~\cite{taghibakhshi2022learning,taghibakhshi2023mggnn}.
In MG, ML has been used, for instance, to construct problem-adaptive smoothers~\cite{li2026level,lyu2026multigrid,wang2025momentum,xie2025mgcfnn} or transfer operators~\cite{holguin2024accelerating,li2025neural,wang2023learning,weymouth2022data}.
A particularly active research direction combines ML models with classical relaxation methods.
While the relaxation methods reduce high-frequency error components, the ML models target the remaining low-frequency components, effectively forming a two-level or multilevel hybrid method~\cite{azulay2022multigrid,cui2022fourier,cui2025hybrid,cui2025neural,kopanivcakova2024deeponet,lee2026hybrid,lee2025fast,sun2025learning,zhang2024blending}.

For nonlinear problems, most existing approaches use ML to provide an initial guess~\cite{aghili2025accelerating,eshaghi2025nows,jin2025fast}, including recent work on spectrally safe neural-operator warm starts for large-scale Newton solvers~\cite{oh2026spectrally}.
Approaches that directly intertwine the nonlinear methods with ML remain comparatively rare. 
Exceptions include for instance POD-based preconditioning of the linearized Newton system~\cite{luo2023preconditioned}, and neural-operator-preconditioned Newton methods~\cite{lee2025neural}.
In addition, the authors of~\cite{zhang2026domain} recently proposed online learning for identifying nonlinear elimination subsets within a parallel DD framework.
These recent developments demonstrate the potential of ML-enhanced nonlinear methods.
However, the systematic integration of ML into nonlinear optimization with convergence guarantees remains largely unexplored.

Moreover, despite the growing literature on ML-accelerated iterative methods, their theoretical understanding remains limited.
Existing analyses have primarily focused on linear problems, including hybrid neural-operator (NO) methods combined with classical relaxation schemes~\cite{cui2025hybrid,hu2025hybrid}.
In particular, these works establish convergence properties by exploiting the spectral structure of the underlying linear iteration.
Even less is understood when learned corrections are incorporated into nonlinear optimization methods, where they may fail to provide descent direction and can therefore compromise the convergence of the underlying solver.
This motivates the development of principled safeguards and convergence guarantees for learned nonlinear corrections, which is a central focus of this work.

In particular, we propose the Safeguarded Neurally-Accelerated Preconditioning (SNAP) framework for hybridizing nonlinear optimization methods with NOs.
The SNAP framework is designed for parametric problems of the form~\eqref{eq:dp_merit}, where the same optimization problem must be solved repeatedly for many parameter instances.
To this end, the NO learns shared structure across the solution manifold ${\{\uv^\star(\boldsymbol{\lambda}) : \boldsymbol{\lambda}\in\Lambda\}}$, which enables generalization to unseen parameters and allows the one-time training cost to be amortized over repeated solves. 

SNAP's algorithmic design is inspired by the multilevel LS method MG/OPT~\cite{lewis2005model,nash2014properties}, which uses search directions from low-resolution or low-fidelity models to accelerate optimization of a high-resolution or high-fidelity model.
In the same spirit, SNAP uses corrections generated by a pre-trained NO to speed up optimization of a given problem. 
However, unlike low-resolution or low-fidelity models, the NO's model does not generally preserve first-order consistency~\cite{kopanicakova2022use} with the original optimization problem.
As a result, the NO-induced corrections are not guaranteed to be a descent direction.
SNAP therefore employs an a posteriori safeguard that accepts learned corrections only when they satisfy prescribed conditions.

This construction lets us establish global convergence of the SNAP framework independently of the chosen NO architecture, while ensuring that accepted NO corrections never increase the objective function value.
We also analyze how the accuracy of the learned corrections affects the reduction in the objective function and relate this accuracy to the NO training error, establishing a direct connection between the learning loss and SNAP's convergence accuracy.

Our numerical experiments demonstrate the effectiveness of the SNAP framework across a range of nonlinear parametric problems, including minimal surface, nonlinear Helmholtz, Allen--Cahn, and jumping-coefficient $p$-Laplacian problems.
We evaluate SNAP, configured with two NO architectures (Fourier neural operator (FNO) and Transolver), both as standalone iterative methods and as nonlinear preconditioners for nonlinear GMRES.
Our numerical results show a substantial reduction in the number of required iterations, while demonstrating robustness and generalization across a wide range of problem parameters and mesh resolutions.

The remainder of this paper is organized as follows.
Section~\ref{sec:SNAP} introduces the SNAP framework, its algorithm, and convergence properties.
Section~\ref{sec:onet_precond} presents the two NO architectures considered in this work, FNO and Transolver, and describes how they are integrated into SNAP and trained.
Afterward, Section~\ref{sec:impl} presents the employed benchmark problems and implementation details, while Section~\ref{sec:num_results} reports the numerical results.
Finally, Section~\ref{sec:conclusion} summarizes the main findings and discusses directions for future work.
\section{Safeguarded Neurally-Accelerated Preconditioning (SNAP) framework}
\label{sec:SNAP}

\begin{figure}[t]
\centering
\includegraphics{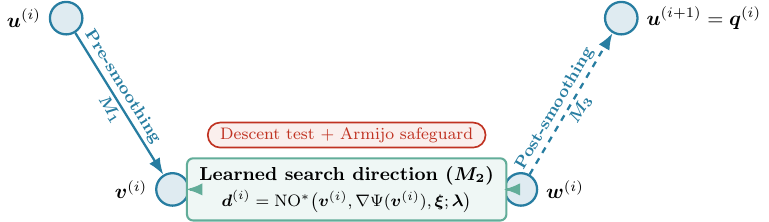}
\caption{
Schematic representation of one SNAP cycle.
Pre-smoothing ($M_1$) is followed by an NO-generated search direction ($M_2$) equipped with an a posteriori safeguard and line search, and an optional post-smoothing step ($M_3$).
}
\label{fig:snap_cycle}
\end{figure}

In this work, we propose the Safeguarded Neurally-Accelerated Preconditioning (SNAP) framework for efficiently solving the minimization problem~\eqref{eq:dp_merit}.
SNAP augments a nonlinear optimization method of a user choice with search directions generated by a pre-trained NO, resulting in a hybrid optimization scheme.
This proposed framework is inspired by the multilevel LS method MG/OPT~\cite{lewis2005model,nash2014properties}, which uses low-resolution or low-fidelity models to accelerate convergence of a given optimization scheme. 
In contrast to MG/OPT, SNAP replaces low-resolution or low-fidelity models with a pre-trained NO.
However, unlike the models used in MG/OPT, the NO does not, in general, preserve first-order consistency with the original objective function~\cite{kopanicakova2022use}.
Consequently, the classical MG/OPT convergence mechanisms for ensuring descent cannot be applied directly.
Instead, SNAP employs an a posteriori safeguard together with a LS globalization strategy to retain convergence of the hybrid iteration.

For the description of SNAP, we fix $\boldsymbol{\lambda}\in\Lambda$ and suppress its dependence from the notation, writing $\Psi(\uv)$ for the objective and $\uv^\star$ for a local minimizer.
The parametric dependence is retained through the NO, which is trained across $\Lambda$ as described in Section~\ref{subsec:training_ons}.
At each $i$-th iteration, SNAP updates the current iterate $\uv^{(i)}\in\R^n$ through a composite step consisting of three stages: pre-smoothing, learned search direction, and post-smoothing.
These stages are represented by the abstract maps $M_j(\nabla\Psi,\uv)\in\R^n$, $j\in\{1,2,3\}$, each of which takes the current iterate and gradient information and returns a search direction.

Starting from $\uv^{(i)}$, a few pre-smoothing iterations based on $M_1$ are first performed with suitable step sizes, producing an intermediate iterate $\vv^{(i)} \in\R^n$.
Here, we use the term \emph{pre-smoothing} following the MG terminology underlying MG/OPT, where this step is traditionally used to reduce high-frequency error components before applying a coarse-level search direction.
In SNAP, the primary role of $M_1$ is to reduce the value of the objective function $\Psi$ before applying the NO's search direction $M_2$. 
However, in practice, $M_1$ is also freqently chosen to effectively reduce high-frequency error components.
The precise requirements on $M_1$ are stated in Assumptions~\ref{ass:smoother_mono} and~\ref{ass:smoother_dec} below and can, for example, be satisfied by gradient-based methods with a LS strategy.

Using the pre-smoothed iterate $\vv^{(i)}$, the pre-trained NO generates the candidate search direction
$\dv^{(i)}$ as
\begin{equation}
\dv^{(i)}
:=
M_{2}(\nabla\Psi,\vv^{(i)})
=
\mathrm{NO}^{\ast}
\bigl(\vv^{(i)},\nabla\Psi(\vv^{(i)}),\boldsymbol{\xi};\boldsymbol{\lambda}\bigr),
\label{eq:dp_dir}
\end{equation}
where $\mathrm{NO}^{\ast}$ denotes the trained NO.
Here, the symbol $\boldsymbol{\xi}$ stands for the nodal coordinates corresponding the iterate, and $\boldsymbol{\lambda}$ denotes the problem parameters.
As will be discussed in Section~\ref{sec:onet_precond}, this search direction requires neither the assembly nor the inversion of the Hessian and can therefore be performed at low computational cost.

However, as the quality of the NO's search direction is not known a priori, it must be assessed \emph{a posteriori}.
To this end, we first verify that $\dv^{(i)}$ is a descent direction, i.e., that~$\langle \nabla\Psi(\vv^{(i)}),\dv^{(i)}\rangle<0$.
If this condition is not satisfied, the search direction is rejected and we set~$\wv^{(i)}=\vv^{(i)}$.
Otherwise, the step size $\eta_2^{(i)}$ is determined by a backtracking LS satisfying the standard Armijo condition~\cite{nocedal2006numerical}, given as
\begin{equation}
\Psi(\vv^{(i)}+\eta_2^{(i)}\dv^{(i)})
\le
\Psi(\vv^{(i)})
+
c_{1}\,\eta_2^{(i)}\,\langle \nabla\Psi(\vv^{(i)}),\dv^{(i)}\rangle,
\qquad
\eta_2^{(i)}\in\{1,\beta,\beta^{2},\dots\}.
\label{eq:dp_accept}
\end{equation}
Both the descent test and the acceptance criterion~\eqref{eq:dp_accept} require only evaluations of the objective function and its gradient, hence they are Hessian-free.
Furthermore, since the accepted direction satisfies $\langle\nabla\Psi(\vv^{(i)}),\dv^{(i)}\rangle<0$, every accepted step yields a strict decrease in the objective function $\Psi$.

The composite SNAP step is concluded by an optional post-smoothing step using $M_3$ with the step size $\eta_3^{(i)}$.
The update produced by $M_3$ is required to not increase the objective.
In practice, $M_3$ and $M_1$ are often chosen to be the same.
The complete SNAP procedure is summarized in Algorithm~\ref{alg:mgopt}.

\begin{algorithm}[htb!]
\caption{SNAP($\Psi, \uv^{(i)}, M_1, M_2, M_3$)}
\label{alg:mgopt}
\begin{algorithmic}[1]
\Require $\Psi:\R^n \rightarrow \R$; $\uv^{(i)} \in \R^n$; $M_{1}$; $M_{2}$; $M_{3}$; smoothing step sizes $\eta_{1}^{(i)}, \eta_{3}^{(i)}$
\Constants $c_{1},\beta\in(0,1)$; $\eta_{\min}\in(0,1]$

\State $\vv^{(i)} \gets \uv^{(i)} + \eta_{1}^{(i)}\,M_{1}(\nabla \Psi, \uv^{(i)})$ \COMMENTmine{Pre-smoothing}
\State $\dv^{(i)} \gets M_{2}(\nabla \Psi, \vv^{(i)})$ \COMMENTmine{NO-induced search direction}
\If{$\langle \nabla\Psi(\vv^{(i)}),\,\dv^{(i)}\rangle < 0$} \COMMENTmine{Descent-direction test}
    \State $\eta_2^{(i)} \gets 1$
    \While{$\eta_2^{(i)} \ge \eta_{\min}$ \textbf{and} $\Psi(\vv^{(i)}+\eta_2^{(i)}\dv^{(i)}) > \Psi(\vv^{(i)}) + c_{1}\,\eta_2^{(i)}\,\langle \nabla\Psi(\vv^{(i)}),\dv^{(i)}\rangle$}
        \State $\eta_2^{(i)} \gets \beta\,\eta_2^{(i)}$ \COMMENTmine{Armijo backtracking}
    \EndWhile
    \If{$\eta_2^{(i)} \ge \eta_{\min}$}
        \State $\wv^{(i)} \gets \vv^{(i)} + \eta_2^{(i)}\,\dv^{(i)}$ \COMMENTmine{Acceptance of $\dv^{(i)}$ with $\eta_2^{(i)}$}
    \Else
        \State $\wv^{(i)} \gets \vv^{(i)}$ \COMMENTmine{Rejection of NO's search direction}
    \EndIf
\Else
    \State $\wv^{(i)} \gets \vv^{(i)}$ \COMMENTmine{Rejection of NO's search direction}
\EndIf
\State $\uv^{(i+1)} = \qv^{(i)} \gets \wv^{(i)} + \eta_{3}^{(i)}\,M_{3}(\nabla \Psi, \wv^{(i)})$ \COMMENTmine{Post-smoothing}
\State \Return $\uv^{(i+1)}$
\end{algorithmic}
\end{algorithm}

\begin{remark}
For quadratic programming problems, when $M_1$ is a standard smoother, such as Jacobi or Gauss--Seidel, and $M_2$ is represented by a DeepONet~\cite{lu2021learning}, SNAP can be seen as a safeguarded, globalized extension of HINTS~\cite{zhang2024blending} and of direct-preconditioning (DP) framework proposed in~\cite{kopanivcakova2024deeponet}.
\end{remark}

\subsection{Global convergence and baseline convergence rate}
To analyze the convergence behavior of SNAP, we consider the following standing assumptions.

\begin{assumption}
\label{ass:smooth}
The function $\Psi$ is twice continuously differentiable on an open convex set $D\subseteq\R^{n}$ containing the iterates, and $\uv^{\star}\in D$ satisfies $\nabla\Psi(\uv^{\star})=\boldsymbol{0}$.
\end{assumption}

\begin{assumption}
\label{ass:lip}
The Hessian $\nabla^{2}\Psi$ is Lipschitz continuous with constant $L$ on a convex neighborhood $N\subseteq D$ of $\uv^{\star}$.
\end{assumption}

\begin{assumption}
\label{ass:spd}
$\nabla^{2}\Psi(\uv^{\star})$ is positive definite and, shrinking $N$ if necessary,
\begin{equation}
0\ <\ \underline{\lambda}\,\|\zv\|_{2}^{2}\ \le\ \langle \nabla^{2}\Psi(\yv)\,\zv,\zv\rangle\ \le\ \overline{\lambda}\,\|\zv\|_{2}^{2}
\qquad\text{for all }\yv\in N,\ \zv\in\R^{n}.
\label{eq:dp_spectral}
\end{equation}
Moreover, the states $\uv^{(i)},\vv^{(i)},\wv^{(i)}$ generated by Algorithm~\ref{alg:mgopt} are assumed to remain in $N$.
\end{assumption}

In addition, we require the following assumptions on smothers $M_1$ and $M_3$. 

\begin{assumption}
\label{ass:smoother_mono}
The step lengths $\eta_{1}^{(i)},\eta_{3}^{(i)}$ in Algorithm~\ref{alg:mgopt} are chosen such that the pre- and post-smoothing steps, performed using $M_{1}$ and $M_{3}$, do not increase the objective $\Psi$, e.g., by a line search that admits the zero step length.
\end{assumption}

\begin{assumption}
\label{ass:smoother_dec}
There exists a fixed $\gamma>0$ such that each pre-smoothing step provides the sufficient decrease
\begin{equation}
\Psi\big(\uv^{(i)}+\eta_{1}^{(i)} M_{1}(\nabla\Psi,\uv^{(i)})\big)\ \le\ \Psi(\uv^{(i)})-\gamma\,\big\|\nabla\Psi(\uv^{(i)})\big\|_{2}^{2}.
\label{eq:dp_smoother_decrease}
\end{equation}
\end{assumption}
Assumption~\ref{ass:smoother_dec} can be fullfilled for example by utilizing LS smoothers with gradient-based search directions~\cite{nocedal2006numerical}.

Let us write $\ev:=\uv^{\star}-\vv$ for the error at an iterate $\vv$.
Furthermore, we define the \emph{optimality gap} at $\vv$ as
\begin{equation}
G(\vv)\ :=\ \Psi(\vv)-\Psi(\uv^{\star}), \quad \forall \vv \in N.
\label{eq:dp_gap}
\end{equation}

Expanding $\Psi$ near $\uv^{\star}$ and using $\nabla\Psi(\uv^{\star})=\boldsymbol{0}$, Assumption~\ref{ass:lip} gives
\begin{equation}
G(\vv)\ =\ \tfrac12\,\|\vv-\uv^{\star}\|_{\Hv^{\star}}^{2}\ +\ R(\vv),
\qquad
|R(\vv)|\ \le\ \tfrac{L}{6}\,\|\vv-\uv^{\star}\|_{2}^{3},
\label{eq:dp_merit_expansion}
\end{equation}
where  $\Hv^{\star}:=\nabla^{2}\Psi(\uv^{\star})$.
The integral form of the same expansion together with the spectral bounds~\eqref{eq:dp_spectral} yields the exact two-sided estimate
\begin{equation}
\tfrac12\,\underline{\lambda}\,\|\vv-\uv^{\star}\|_{2}^{2}\ \le\ G(\vv)\ \le\ \tfrac12\,\overline{\lambda}\,\|\vv-\uv^{\star}\|_{2}^{2},
\qquad \forall \vv \in N.
\label{eq:dp_gap_bounds}
\end{equation}
Thus, $G\ge0$ on $N$.

In addition, our analysis requires the gradient to dominate the optimality gap.
Thus, we also utilize the following Polyak--{\Lpolish}ojasiewicz (PL) inequality, introduced by Polyak~\cite{polyak1963gradient} for nonconvex functions and widely used to establish linear convergence rates of first-order descent methods without convexity~\cite{karimi2016linear}; see also~\cite[Section~3.8]{wright2022optimization}.

\begin{assumption}
\label{ass:pl}
There exists $\mu>0$ such that
\begin{equation}
\tfrac12\,\big\|\nabla\Psi(\vv)\big\|_{2}^{2}\ \ge\ \mu\,G(\vv)
\qquad \forall \vv \in N .
\label{eq:dp_pl}
\end{equation}
\end{assumption}
Note, under Assumptions~\ref{ass:lip}--\ref{ass:spd}, the integral form of the Taylor expansion gives $\|\nabla\Psi(\vv)\|_{2}\ge\underline{\lambda}\,\|\vv-\uv^{\star}\|_{2}$ on $N$, which combined with the upper bound in~\eqref{eq:dp_gap_bounds} yields~\eqref{eq:dp_pl} with $\mu=\underline{\lambda}^{2}/\overline{\lambda}$.
The assumption~\eqref{eq:dp_pl} is, however, strictly weaker than Assumption~\ref{ass:spd}, as it also allows non-isolated minimizers.

Note, under Assumption~\ref{ass:spd}, $\Psi$ is strongly convex on the convex set $N\ni\uv^{\star}$, so that for all $\vv\in N$
\begin{equation*}
\Psi(\uv^{\star})\ \ge\ \Psi(\vv)+\langle\nabla\Psi(\vv),\uv^{\star}-\vv\rangle+\tfrac{\underline{\lambda}}{2}\|\uv^{\star}-\vv\|_{2}^{2}\ \ge\ \Psi(\vv)-\tfrac{1}{2\underline{\lambda}}\|\nabla\Psi(\vv)\|_{2}^{2},
\end{equation*}
i.e., \eqref{eq:dp_pl} holds with $\mu=\underline{\lambda}$.
Assumption~\ref{ass:pl} is therefore implied by Assumption~\ref{ass:spd}.

\begin{lemma}
\label{lem:dp_global}
Let Assumptions~\ref{ass:smooth} and~\ref{ass:smoother_mono} hold, and let each iteration take the step of Algorithm~\ref{alg:mgopt}.
Then, for any NO oracle $M_{2}$, the sequence $\{\Psi(\uv^{(i)})\}_{i\ge0}$ is non-increasing.
If, in addition, $\Psi$ is bounded from below on $D$ and Assumption~\ref{ass:smoother_dec} holds, then $\nabla\Psi(\uv^{(i)})\to0$ as $i\to\infty$, independently of the direction quality.
\end{lemma}

\begin{proof}
We first show that no step of Algorithm~\ref{alg:mgopt} increases the objective.
An accepted search direction satisfies~\eqref{eq:dp_accept} with $\eta_{2}^{(i)}\in[\eta_{\min},1]$ and $\langle \nabla\Psi(\vv^{(i)}),\dv^{(i)}\rangle<0$, whence ${\Psi(\wv^{(i)})\le\Psi(\vv^{(i)})+c_{1}\,\eta_{2}^{(i)}\,\langle \nabla\Psi(\vv^{(i)}),\dv^{(i)}\rangle<\Psi(\vv^{(i)})}$. 
A rejected search direction leaves the iterate unchanged, i.e., $\wv^{(i)}=\vv^{(i)}$ and the smoother steps do not increase $\Psi$ by Assumption~\ref{ass:smoother_mono}.
Consequently,
$\Psi(\uv^{(i+1)})=\Psi(\qv^{(i)})\le\Psi(\wv^{(i)})\le\Psi(\vv^{(i)})\le\Psi(\uv^{(i)})$ for every $i$, which is the claimed monotonicity.
Under Assumption~\ref{ass:smoother_dec}, the pre-smoothing step alone gives $\Psi(\uv^{(i+1)})\le\Psi(\vv^{(i)})\le\Psi(\uv^{(i)})-\gamma\,\|\nabla\Psi(\uv^{(i)})\|_{2}^{2}$.
Summing over $i=0,\dots,N$ and telescoping,
\begin{equation*}
\gamma\sum_{i=0}^{N}\big\|\nabla\Psi(\uv^{(i)})\big\|_{2}^{2}\ \le\ \Psi(\uv^{(0)})-\Psi(\uv^{(N+1)})\ \le\ \Psi(\uv^{(0)})-\inf_{D}\Psi\ <\ \infty ,
\end{equation*}
so the series converges as $N\to\infty$ and its terms vanish, $\nabla\Psi(\uv^{(i)})\to0$.
\end{proof}

\begin{remark}
\label{rem:dp_unbounded}
The second part of Lemma~\ref{lem:dp_global} requires $\Psi$ to be bounded from below on the set $D$ of Assumption~\ref{ass:smooth}, which contains all iterates.
 For energies unbounded from below on $\R^n$, such as those arising in the Bratu-Gelfand and nonlinear Helmholtz problems of Section~\ref{sec:impl}, this condition holds if $D$ is bounded, since $\Psi$ is continuous on the closure of $D$.
 However, monotone decrease of $\Psi$ does not guarantee bounded iterates when the relevant sublevel sets are unbounded.
 Therefore, boundedness must be assumed, as in the choice of $D$, or ensured through additional algorithmic safeguards.
\end{remark}

In practice, 
Lemma~\ref{lem:dp_global} states that the NO's search direction can never impair the robustness of the iteration. 
However, if $\dv$ is poor search direction, the safeguards keep the objective monotone, and first-order convergence is driven by the smoother alone.
Notably, the proof uses no property of the direction $\dv$, thus the lemma is agnostic of a particular choice of NO's architecture.

Under the gradient-dominance assumption, the smoother safeguarding can be stated quantitatively by the following lemma. 

\begin{lemma}
\label{lem:dp_smoother_rate}
Let Assumptions~\ref{ass:smooth}, \ref{ass:spd}, \ref{ass:smoother_mono}, \ref{ass:smoother_dec} and~\ref{ass:pl} hold, and let ${\uv^{(i)},\vv^{(i)}\in N}$.
Then SNAP algorithm, with iteration as specified in Algorithm~\ref{alg:mgopt} satisfies, irrespective of the NO's search direction,
\begin{equation}
G(\uv^{(i+1)})\ \le\ (1-2\gamma\mu)\,G(\uv^{(i)}),
\label{eq:dp_smoother_rate}
\end{equation}
where $(2\gamma\mu)\le1$ whenever $G(\uv^{(i)})>0$.
\end{lemma}

\begin{proof}
Assumption~\ref{ass:smoother_dec} and~\eqref{eq:dp_pl} at $\uv^{(i)}\in N$ give
$\Psi(\vv^{(i)})\le\Psi(\uv^{(i)}) - \gamma\|\nabla\Psi(\uv^{(i)})\|_{2}^{2}\le\Psi(\uv^{(i)}) - 2\gamma\mu\,G(\uv^{(i)})$.
Then,  subtracting $\Psi(\uv^{\star})$ yields ${G(\vv^{(i)})\le(1-2\gamma\mu)\,G(\uv^{(i)})}$.
Since $\vv^{(i)}\in N$, the lower bound in~\eqref{eq:dp_gap_bounds} gives $G(\vv^{(i)})\ge0$, which forces $2\gamma\mu\le1$ whenever $G(\uv^{(i)})>0$.
The remaining moves of the cycle do not increase $\Psi$, as shown in the proof of Lemma~\ref{lem:dp_global}, whence $G(\uv^{(i+1)})\le G(\vv^{(i)})$.
\end{proof}

In practice, 
Lemma~\ref{lem:dp_smoother_rate} provides the baseline that any NO-based search direction must improve on. 
As we can see, the optimality gap contracts linearly, but at a condition-limited rate.
In particular, for gradient-based smoothers, $\gamma$ scales like an inverse Lipschitz constant, i.e., $1/\overline{\lambda}$, and $\mu=\underline{\lambda}$, so $2\gamma\mu$ is of the order $1/\kappa$ with $\kappa=\overline{\lambda}/\underline{\lambda}$, and the factor $(1-2\gamma\mu)$ approaches one for ill-conditioned problems.
\subsection{NO's search direction accuracy and accelerated convergence}

The natural measure of progress for the safeguarded step is the optimality gap~\eqref{eq:dp_gap}.
We therefore quantify the accuracy of the NO's direction by the \emph{defect} $\varepsilon(\vv)$, given as
\begin{equation}
\varepsilon(\vv)\ :=\ \Psi\big(\vv+M_{2}(\nabla\Psi,\vv)\big)-\Psi(\uv^{\star})\ =\ G\big(\vv+M_{2}(\nabla\Psi,\vv)\big),
\label{eq:dp_eps}
\end{equation}
together with its relative counterpart $\theta(\vv):=\varepsilon(\vv)/G(\vv)$, which is well defined whenever $G(\vv)>0$.
Note that $\varepsilon$ is deliberately evaluated at the \emph{unit} step, i.e., it measures the quality of the search direction as returned by the oracle, i.e., independently of the step-size control.
Here, we highlight that the defect $\varepsilon(\vv)$ vanishes if and only if the search direction is exact, i.e., $\uv^\star = \vv+\dv$.
Moreover, $\Psi(\vv+\dv)$ is precisely the quantity that the Armijo test~\eqref{eq:dp_accept} evaluates at its first trial step.

\begin{proposition}
\label{prop:dp}
Let Assumptions~\ref{ass:smooth}, \ref{ass:spd}, \ref{ass:smoother_mono}, \ref{ass:smoother_dec} and~\ref{ass:pl} hold, and consider the $i$-th cycle of Algorithm~\ref{alg:mgopt} with $G(\uv^{(i)})>0$, whose states $\uv^{(i)}$, $\vv^{(i)}$ and $\wv^{(i)}$ lie in $N$. Then:
\begin{enumerate}
\item[(i)] If the NO's search direction is rejected, then
${G(\uv^{(i+1)})\ \le\ (1-2\gamma\mu)\,G(\uv^{(i)})}$.
\item[(ii)] If the NO's search direction is accepted with a backtracked step $\eta_{2}^{(i)}<1$, then
${G(\uv^{(i+1)})\ \le\ (1-2\gamma\mu)\,G(\uv^{(i)})\ -\ c_{1}\,\eta_{2}^{(i)}\,\big|\langle\nabla\Psi(\vv^{(i)}),\dv^{(i)}\rangle\big|}$.
\item[(iii)] If the search direction is accepted with $\eta_{2}^{(i)}=1$, then $\theta(\vv^{(i)})<1$ and
\begin{equation}
G(\uv^{(i+1)})\ \le\ (1-2\gamma\mu)\ \theta(\vv^{(i)})\ G(\uv^{(i)}) .
\label{eq:dp_gap_rate}
\end{equation}
\end{enumerate}
Consequently, the safeguarded cycle is never slower than $M_1$ alone, and each unit-step NO's search direction multiplies the gap by the additional factor $\theta(\vv^{(i)})$.
Thus, the smaller $\theta(\vv^{(i)})$, the faster the convergence.
\end{proposition}

\begin{proof}
From the proof of Lemma~\ref{lem:dp_smoother_rate}, the pre-smoothing step satisfies 
\begin{equation*}
    {G(\vv^{(i)})\le(1-2\gamma\mu)\,G(\uv^{(i)})},
\end{equation*}
and post-smoothing does not increase $\Psi$ by Assumption~\ref{ass:smoother_mono}.
If the NO's search direction is rejected, $\wv^{(i)}=\vv^{(i)}$, which gives claim~(i).

If NO's search direction $\dv^{(i)}$ it is accepted with step size $\eta_{2}^{(i)}$, the Armijo condition~\eqref{eq:dp_accept} and $\langle\nabla\Psi(\vv^{(i)}),\dv^{(i)}\rangle<0$ give 
\begin{equation*}
    \Psi(\wv^{(i)})\le\Psi(\vv^{(i)})-c_{1}\eta_{2}^{(i)}\,|\langle\nabla\Psi(\vv^{(i)}),\dv^{(i)}\rangle|,
\end{equation*}
i.e., $G(\wv^{(i)})\le G(\vv^{(i)})-c_{1}\eta_{2}^{(i)}\,|\langle\nabla\Psi(\vv^{(i)}),\dv^{(i)}\rangle|$, which gives claim~(ii).

If NO's search direction $\dv^{(i)}$ it is accepted with the unit step $\eta_{2}^{(i)}=1$, then the descent test gives $\nabla\Psi(\vv^{(i)})\neq\boldsymbol{0}$, so $\vv^{(i)}\neq\uv^{\star}$ and $G(\vv^{(i)})>0$ by the lower bound in~\eqref{eq:dp_gap_bounds}.
Hence $\theta(\vv^{(i)})$ is well defined.
Moreover, $\wv^{(i)}=\vv^{(i)}+\dv^{(i)}$, so ${G(\wv^{(i)})=\varepsilon(\vv^{(i)})=\theta(\vv^{(i)})\,G(\vv^{(i)})}$ by the definition~\eqref{eq:dp_eps}, and the strict Armijo decrease gives $\theta(\vv^{(i)})<1$, while $\wv^{(i)}\in N$ and~\eqref{eq:dp_gap_bounds} give $\theta(\vv^{(i)})\ge0$.
Multiplying the pre-smoothing bound by $\theta(\vv^{(i)})\ge0$ yields~\eqref{eq:dp_gap_rate} and claim~(iii).
\end{proof}

In practical terms, Proposition~\ref{prop:dp} separates the two convergence mechanisms multiplicatively.
The smoother contributes its factor $(1-2\gamma\mu)$ in every cycle, while the NO's search direction contributes $\theta(\vv)$ exactly when it is accurate.

Cases~(i)--(iii) of Proposition~\ref{prop:dp} are \emph{a posteriori} statements. 
They quantify the progress once the outcome of the safeguard is known, but they do not predict which case occurs.
The following corollary supplies the missing \emph{a priori} information specifying that sufficiently accurate NO's search direction forces the fast case~(iii), so an accurate NO's search direction is never wasted.

\begin{corollary}
\label{cor:dp_accept}
Let, in addition to the assumptions of Proposition~\ref{prop:dp}, Assumption~\ref{ass:lip} hold.
Now, let $\vv^{(i)}+\dv^{(i)}\in N$, and 
\begin{equation}
\theta(\vv^{(i)})\ \le\ \bar\theta\ <\ \frac{1}{4\kappa^{2}}
\qquad\text{and}\qquad
c_{1}\ \le\ \frac{1-\bar\theta}{2\kappa\,(1+\sqrt{\bar\theta}\,)},
\qquad \kappa:=\overline{\lambda}/\underline{\lambda}.
\label{eq:dp_accept_cond}
\end{equation}
Then case~(iii) of Proposition~\ref{prop:dp} occurs, i.e., NO's search direction passes the descent test and satisfies the Armijo rule~\eqref{eq:dp_accept} already at the first trial step $\eta_{2}=1$.
\end{corollary}

\begin{proof}
We abbreviate $\vv=\vv^{(i)}$ and $\dv=\dv^{(i)}$.
By Assumption~\ref{ass:spd}, $\Psi$ is strongly convex on $N$, so $\Psi(\uv^{\star})\ge\Psi(\vv)+\langle\nabla\Psi(\vv),\uv^{\star}-\vv\rangle$ gives $\langle\nabla\Psi(\vv),\ev\rangle\le-G(\vv)$, while~\eqref{eq:dp_gap_bounds} yields ${\|\ev\|_{2}\le\sqrt{2G(\vv)/\underline{\lambda}}}$ and $\|\dv-\ev\|_{2}\le\sqrt{2\varepsilon(\vv)/\underline{\lambda}}$, and the integral form of the Taylor expansion gives ${\|\nabla\Psi(\vv)\|_{2}\le\overline{\lambda}\,\|\ev\|_{2}}$.
The Cauchy--Schwarz inequality then bounds the slope from both sides,
\begin{equation*}
-2\kappa\big(1+\sqrt{\bar\theta}\,\big)\,G(\vv)\ \le\ \langle\nabla\Psi(\vv),\dv\rangle\ \le\ -\big(1-2\kappa\sqrt{\bar\theta}\,\big)\,G(\vv) .
\end{equation*}
The upper bound is negative for $\bar\theta<1/(4\kappa^{2})$, so the descent test passes.
The Armijo rule at $\eta_{2}=1$ reads $\Psi(\vv+\dv)-\Psi(\vv)=-(1-\theta(\vv))\,G(\vv)\le c_{1}\langle\nabla\Psi(\vv),\dv\rangle$, which holds by the lower bound whenever $c_{1}\le(1-\bar\theta)/\big(2\kappa(1+\sqrt{\bar\theta}\,)\big)$.
\end{proof}

Corollary~\ref{cor:dp_accept} shows that a sufficiently accurate NO's search direction is guaranteed to pass both the descent test and the Armijo condition with the unit step $\eta_2=1$.
Here, the required accuracy is relative to the current optimality gap $G(\vv)$.
Thus, if the gap is large, a larger absolute correction error can be tolerated, whereas increasingly accurate corrections are required as the iterate approaches the solution.
The sufficient accuracy requirement also becomes more stringent as the local condition number $\kappa$ increases, i.e., as the problem becomes more ill-conditioned.
The constants in~\eqref{eq:dp_accept_cond} are conservative, as the dependence on $\kappa$ arises from the Euclidean Cauchy--Schwarz bounds used in the analysis, e.g., $\kappa=10^{3}$ requires $\bar\theta<2.5\cdot10^{-7}$.
Since $\sqrt{\bar{\theta}}<1/(2\kappa)$, the condition on the Armijo parameter is implied by $c_1\le1/(4\kappa)$, which is restrictive for ill-conditioned problems.
If the conditions in~\eqref{eq:dp_accept_cond} are not met, cases~(i)--(ii) of Proposition~\ref{prop:dp} still hold.

\subsection{Impact of training on solver performance}
\label{sec:impact_training}
We now relate the training loss of the NO to the convergence behavior of SNAP.

\begin{proposition}
\label{cor:dp_training}
Let Assumptions~\ref{ass:smooth} and~\ref{ass:smoother_mono} hold, and consider the $i$-th cycle of Algorithm~\ref{alg:mgopt}.
Then
\begin{equation}
G(\uv^{(i+1)})\ \le\ G(\wv^{(i)})\ \le\ G(\vv^{(i)}),
\label{eq:dp_training_gap}
\end{equation}
where $G(\wv^{(i)})=\varepsilon(\vv^{(i)})$ whenever the NO's search direction is accepted with the unit step $\eta_{2}^{(i)}=1$.
Under the assumptions of Lemma~\ref{lem:dp_smoother_rate}, $G(\vv^{(i)})$ in~\eqref{eq:dp_training_gap} can be replaced by $(1-2\gamma\mu)\,G(\uv^{(i)})$.
Moreover, let the assumptions of Corollary~\ref{cor:dp_accept} hold, except for the condition $\theta(\vv^{(i)})\le\bar\theta$, i.e., let $\bar\theta<1/(4\kappa^{2})$ and $c_{1}$ satisfy~\eqref{eq:dp_accept_cond}, while $\theta(\vv^{(i)})$ is arbitrary.
Then $\varepsilon(\vv^{(i)})\ge0$ and
\begin{equation}
G(\wv^{(i)})\ \le\
\begin{cases}
\varepsilon(\vv^{(i)}), & \varepsilon(\vv^{(i)})\le\bar\theta\,G(\vv^{(i)}),\\[2pt]
G(\vv^{(i)})\ <\ \varepsilon(\vv^{(i)})/\bar\theta, & \varepsilon(\vv^{(i)})>\bar\theta\,G(\vv^{(i)}),
\end{cases}
\qquad
G(\uv^{(i+1)})\ \le\ \frac{\varepsilon(\vv^{(i)})}{\bar\theta}.
\label{eq:dp_training_bound}
\end{equation}
Consequently, if these assumptions hold for all states $\vv$ of a distribution $\pazocal{S}$ of pre-smoothed states, then
\begin{equation}
{\mathbb{E}_{\vv\sim\pazocal{S}}[G(\uv^{(i+1)})]\ \le\ }\mathbb{E}_{\vv\sim\pazocal{S}}[G(\wv(\vv))]\ \le\ \min\Big\{\mathbb{E}_{\vv\sim\pazocal{S}}[G(\vv)],\ \tfrac{1}{\bar\theta}\,\mathbb{E}_{\vv\sim\pazocal{S}}[\varepsilon(\vv)]\Big\},
\label{eq:dp_expected_bound}
\end{equation}
where the expected defect is
\begin{equation}
\mathbb{E}_{\vv\sim\pazocal{S}}\big[\varepsilon(\vv)\big]
=
\mathbb{E}_{\vv\sim\pazocal{S}}
\Big[\Psi\big(\vv+\dv(\vv)\big)\Big]
-
\Psi(\uv^{\star}).
\label{eq:dp_optimal_loss}
\end{equation}
Moreover, if Assumptions~\ref{ass:lip} and~\ref{ass:spd} hold and ${\vv+\dv\in N}$, the 
bound~\eqref{eq:dp_gap_bounds} gives
\begin{equation}
\tfrac12\,\underline{\lambda}\,\|\dv-\ev\|_{2}^{2}
\le
\varepsilon(\vv)
\le
\tfrac12\,\overline{\lambda}\,\|\dv-\ev\|_{2}^{2},
\label{eq:dp_eps_bounds}
\end{equation}
and the expansion~\eqref{eq:dp_merit_expansion} yields
\begin{equation}
\varepsilon(\vv)
=
\tfrac12\|\dv-\ev\|_{\Hv^{\star}}^{2}
+
\pazocal{O}\big(\|\dv-\ev\|_{2}^{3}\big).
\label{eq:dp_eps_expansion}
\end{equation}
Thus, the $\Hv^{\star}$-weighted solution match is the second-order approximation of the objective loss, while the Euclidean solution match controls the defect up to the spectral weighting in~\eqref{eq:dp_eps_bounds}.
\end{proposition}

\begin{proof}
If the unit step is accepted, then $\wv^{(i)}=\vv^{(i)}+\dv^{(i)}$, so $G(\wv^{(i)})=\varepsilon(\vv^{(i)})$ by the definition~\eqref{eq:dp_eps}.
Otherwise, either a backtracked step satisfying the Armijo condition~\eqref{eq:dp_accept} is taken or the search direction is rejected, and in both cases ${G(\wv^{(i)})\le G(\vv^{(i)})}$.
Since post-smoothing does not increase $\Psi$ by Assumption~\ref{ass:smoother_mono}, ${G(\uv^{(i+1)})\le G(\wv^{(i)})}$, which gives~\eqref{eq:dp_training_gap}.
The variant with ${(1-2\gamma\mu)\,G(\uv^{(i)})}$ follows from the pre-smoothing bound in the proof of Lemma~\ref{lem:dp_smoother_rate}
The lower bound $\varepsilon(\vv^{(i)})\ge0$ follows from~\eqref{eq:dp_gap_bounds}, as $\vv^{(i)}+\dv^{(i)}\in N$.
If $\varepsilon(\vv^{(i)})\le\bar\theta\,G(\vv^{(i)})$, Corollary~\ref{cor:dp_accept} guarantees that the unit step is accepted, so $G(\wv^{(i)})=\varepsilon(\vv^{(i)})$.
Otherwise, ${G(\wv^{(i)})\le G(\vv^{(i)})<\varepsilon(\vv^{(i)})/\bar\theta}$ by~\eqref{eq:dp_training_gap}.
Since $\bar\theta<1$, $G(\uv^{(i+1)})\le G(\wv^{(i)})\le\varepsilon(\vv^{(i)})/\bar\theta$ in both cases, which gives~\eqref{eq:dp_training_bound}.
Taking expectations of $G(\uv^{(i+1)})\le G(\wv^{(i)})\le G(\vv^{(i)})$ and of $G(\wv^{(i)})\le\varepsilon(\vv^{(i)})/\bar\theta$ gives~\eqref{eq:dp_expected_bound}.
Taking the expectation of~\eqref{eq:dp_eps} over $\vv\sim\pazocal{S}$ gives~\eqref{eq:dp_optimal_loss}.
The remaining claims follow from~\eqref{eq:dp_gap_bounds} and~\eqref{eq:dp_merit_expansion} evaluated at $\wv=\vv+\dv$, for which $\wv-\uv^{\star}=\dv-\ev$.
\end{proof}

\begin{remark}
\label{rem:dp_training_surrogate}
The loss~\eqref{eq:dp_optimal_loss} serves as a surrogate for the progress of the safeguarded SNAP cycle.
By~\eqref{eq:dp_expected_bound}, the expected correction defect controls the expected post-cycle optimality gap, but only through the factor $1/\bar{\theta} > 4 \kappa^{2}$, which is pessimistic. 
The bound is tight on states where the unit step is accepted, since there $G(\wv)=\varepsilon(\vv)$, whereas on the remaining states the safeguard guarantees only the smoother baseline $G(\vv)$, and $\varepsilon(\vv)$ may be far larger than the actual gap.
Moreover, \eqref{eq:dp_optimal_loss} relates to the solver only if $\pazocal{S}$ approximates the distribution of pre-smoothed states visited by Algorithm~\ref{alg:mgopt}, which itself depends on the trained NO.
\end{remark}

Proposition~\ref{cor:dp_training} provides several practical guidelines for training NOs, such that they yield good search directions: 
\begin{itemize}
    \item \textbf{Training objective:}
	Proposition~\ref{cor:dp_training} suggests to train the NO using the objective function-based loss~\eqref{eq:dp_optimal_loss}, which directly targets the defect governing the reduction achieved by an accepted unit-step search direction.
    Alternatively, one may employ the commonly used Euclidean solution-match loss 
    \begin{equation}
    \|\vv+\dv-\uv^{\star}\|_{2}^{2}
    =
    \|\dv-\ev\|_{2}^{2}.
    \label{eq:dp_solution_match}
    \end{equation}
    The bounds~\eqref{eq:dp_eps_bounds} establish the equivalence of this loss and the defect up to the spectral constants $\underline{\lambda}$ and $\overline{\lambda}$.
    However, the expression~\eqref{eq:dp_eps_expansion} shows that the local quadratic model of the defect is given by the $\Hv^{\star}$-weighted solution error.
    Hence, Euclidean solution matching constitutes a consistent loss, but its effectiveness may deteriorate for ill-conditioned problems. 

    \item \textbf{Search direction magnitude:}
    Training should recover not only a useful search direction but also an appropriate search direction magnitude.
    The backtracking line search in Algorithm~\ref{alg:mgopt} satisfies $\eta_{2}\le1$ and can therefore reduce an overestimated search direction, but it cannot enlarge an underestimated one.
    Accurate scaling of the NO inputs and outputs is of crucial importance for obtaining search directions that can be accepted close to the unit step.
    This motivates the normalization strategy introduced in Section~\ref{subsec:training_ons}.

    \item \textbf{Training-state distribution:}
    The relevant distribution $\pazocal{S}$ is the distribution of pre-smoothed states encountered by Algorithm~\ref{alg:mgopt}.
    Predictions on states outside the support of the training distribution are not directly constrained by the loss and may therefore provide less useful search direction directions.
    The convergence safeguard preserves the baseline convergence mechanism in this case, but the potential acceleration may be lost.
    Training data should therefore approximate the visitation distribution of SNAP. 

    \item \textbf{Sampling strategy:}
    The distribution $\pazocal{S}$ need not be sampled uniformly around $\uv^{\star}$.
    Instead, it should place probability mass on states representative of those encountered during the SNAP iteration.
    For example, one can generate such samples by perturbing reference solutions and then applying a small number of $M_1$ smoothing steps, or by collecting pre-smoothed states from representative solver trajectories.
    The latter might more directly approximate the visitation distribution, whereas the former should provide a simple offline sampling procedure that does not require full SNAP trajectories.
    Sampling should also cover different distances from $\uv^{\star}$, since the magnitude and structure of the desired search direction $\uv^{\star}-\vv$ change throughout the solve.

\item \textbf{Accuracy floor and late-stage convergence:}
In practice, the trained NO may only approximate the ideal correction $\ev$ up to a finite accuracy, such that the correction error $\|\dv-\ev\|_2$ does not decrease substantially below 
$\delta_{\mathrm{floor}}>0$.
By~\eqref{eq:dp_eps_bounds}, this induces a corresponding defect floor $\varepsilon_{\mathrm{floor}}>0$ satisfying
\begin{equation*}
\tfrac{1}{2}\underline{\lambda} \delta_{\mathrm{floor}}^2
\lesssim
\varepsilon_{\mathrm{floor}}
\lesssim
\tfrac{1}{2}\overline{\lambda} \delta_{\mathrm{floor}}^2.
\end{equation*}
As the iteration approaches the solution, the optimality gap $G(\vv)$ decreases, whereas the defect cannot decrease substantially below $\varepsilon_{\mathrm{floor}}$.
Once this accuracy floor is reached, the relative defect therefore behaves as
\begin{equation*}
\theta(\vv)
=
\frac{\varepsilon(\vv)}{G(\vv)}
\approx
\frac{\varepsilon_{\mathrm{floor}}}{G(\vv)}.
\end{equation*}
Corollary~\ref{cor:dp_accept} guarantees the acceptance of the unit step as long as $\varepsilon_{\mathrm{floor}}\le\bar\theta\,G(\vv)$, i.e., $G(\vv)\ge\varepsilon_{\mathrm{floor}}/\bar\theta>4\kappa^2\varepsilon_{\mathrm{floor}}$.
Hence, the learned correction is most effective while $G(\vv)\gg\varepsilon_{\mathrm{floor}}$, but becomes less effective as the optimality gap approaches the defect floor.
SNAP may therefore exhibit an initial accelerated regime followed by a late-stage regime in which the iteration increasingly relies on the underlying optimization method, given by $M_1$.
Reducing the NO prediction-error floor extends the accelerated regime. 

\end{itemize}
\section{Neural operators (NOs) as search direction maps}
\label{sec:onet_precond}

In this work, we construct the search direction map $M_2$ using a pre-trained NO.
Let $\mathrm{NO}_{\thetav}$ denote the NO parameterized by weights $\thetav\in\R^{n_w}$, and let ${\mathrm{NO}^{\ast}:=\mathrm{NO}_{\thetav^\ast}}$ be the trained operator.
Given the current iterate $\uv$, its gradient $\gv:=\nabla\Psi(\uv)$, the nodal coordinates $\boldsymbol{\xi}$, and the problem parameters $\boldsymbol{\lambda}$, the trained NO predicts the search direction
${\dv := \mathrm{NO}^{\ast}(\uv,\gv;\boldsymbol{\xi},\boldsymbol{\lambda}) \in\R^n}$.
The spatial quantities $(\uv, \gv, \boldsymbol{\xi})$ are supplied as nodal inputs, while $\boldsymbol{\lambda}$ is incorporated through a learned parameter embedding.

In this work, we consider two NO architectures, namely FNO~\cite{li2020fourier} and Transolver~\cite{zhou2026transolver}, described below.
For each NO, we provide a description of its architecture. 
Our description follows the standard notation from the literature.
Unless stated otherwise, the symbols introduced in this section are local to it and are unrelated to identically named quantities used elsewhere in the manuscript.

\subsection{Fourier neural operator (FNO)}
\label{sec:fno}
The FNO~\cite{li2020fourier} is a NO architecture that learns mappings between spatial fields using global convolution operators parameterized in Fourier space.
In our setting, the FNO maps the current iterate and its gradient to a search direction.
The nodal inputs are stacked channel-wise into a feature field $a\in\R^{n\times c}$, where $n$ denotes the number of nodal dofs and $c$ the number of input features per node.

The FNO first lifts $a$ pointwise to a width-$w$ latent representation and incorporates the problem parameters through a learned embedding, thus as
\begin{align}
\hv^{0}
=
a\Wv_{0}^{T}
+
\boldsymbol{\lambda}\Ev^{T},
\label{eq:fno_lift}
\end{align}
where $\Wv_{0}\in\R^{w\times c}$ and $\Ev\in\R^{w\times n_{\lambda}}$ are learnable weight matrices.
Here, we point out that the parameter embedding $\boldsymbol{\lambda}\Ev^{T}\in\R^w$ is broadcast over the spatial nodes.

The latent field is then propagated through $n_L$ Fourier layers as
\begin{align}
\hv^{\ell+1}
=
\sigma\left(
\hv^{\ell}\Wv_\ell^{T}
+
\pazocal{K}_\ell(\hv^{\ell})
\right),
\qquad
\ell=0,\ldots,n_L-1,
\label{eq:fno_layer}
\end{align}
where $\sigma$ denotes the activation function, commonly chosen as $\mathrm{GELU}$. 
The weights $\Wv_\ell\in\R^{w\times w}$ act pointwise and the spectral convolution $\pazocal{K}_\ell$ is defined as
\begin{align*}
\pazocal{K}_\ell(\hv)
=
\pazocal{F}^{-1}
\left(
\pazocal{F}(\hv)\Rv_\ell^{T}
\right),
\qquad
\left(
\pazocal{F}(\hv)\Rv_\ell^{T}
\right)(\omega)
=
\begin{cases}
(\pazocal{F}(\hv))(\omega)\Rv_\ell(\omega)^{T},
& |\omega|\le\omega_{\max},
\\
0,
& |\omega|>\omega_{\max}.
\end{cases}
\end{align*}
Here, $\pazocal{F}$ and $\pazocal{F}^{-1}$ denote the discrete Fourier transform and its inverse, and ${\Rv_\ell(\omega)\in\mathbb{C}^{w\times w}}$ are learnable weights retained for the selected Fourier modes.

The final latent field $\hv^{n_L}$ defines the feature representation
$\boldsymbol{\Phi}=\hv^{n_L}\in\R^{n\times w}$, from which a linear readout produces the search direction $\dv$.
The FNO architecture is illustrated in Figure~\ref{fig:fno}.

\begin{remark}
The standard FNO applies multidimensional Fourier transforms on structured grids, although extensions to irregular geometries exist~\cite{duruisseaux2512fourier}.
For finite element discretizations, an alternative is to apply a one-dimensional Fourier transform to the vector of basis coefficients, an approach referred to here as F-FNO.
\end{remark}

\begin{figure}[t]
\centering
\resizebox{\linewidth}{!}{%
\includegraphics{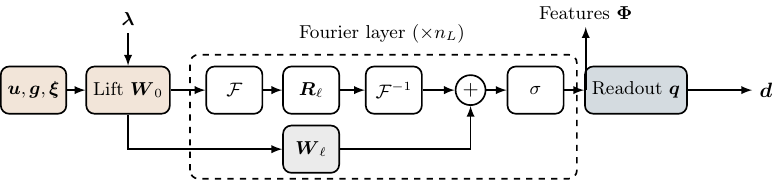}
}
\caption{Architecture of the FNO used to generate the SNAP search direction.}
\label{fig:fno}
\end{figure}

\subsection{Transolver}
\label{sec:transolver}
The second architecture considered in this work is Transolver~\cite{luo2025transolver,wu2024transolver,zhou2026transolver}, for which we employ the Transolver-3 variant~\cite{zhou2026transolver}.
Transolver is a NO architecture that performs self-attention over a reduced set of learned \emph{physical states}, rather than directly over all mesh points.
In our setting, the nodal inputs are stacked channel-wise into a feature field $a\in\R^{n\times c}$, as for the FNO.

The input is first lifted to a width-$w$ latent representation and combined with the problem parameters through a learned embedding, i.e., 
\begin{align}
\xv^{0}
=
\mathrm{MLP}(a)
+
\boldsymbol{\lambda}\Ev^{T},
\label{eq:transolver_lift}
\end{align}
where $\mathrm{MLP}$ denotes the multilayer perceptron.

The latent representation $\xv^{0}$ is then propagated through $n_L$ Transolver blocks, given as
\begin{align}
\xv^{\ell+1/2}
&=
\xv^{\ell}
+
\pazocal{A}\bigl(\mathrm{LN}(\xv^{\ell})\bigr),
\nonumber\\
\xv^{\ell+1}
&=
\xv^{\ell+1/2}
+
\mathrm{MLP}\bigl(\mathrm{LN}(\xv^{\ell+1/2})\bigr),
\qquad
\ell=0,\ldots,n_L-1,
\label{eq:transolver_block}
\end{align}
where $\pazocal{A}$ denotes the physics-attention module and $\mathrm{LN}$ denotes layer normalization.

To describe the physics-attention module~$\pazocal{A}$, let us define $\yv^\ell=\mathrm{LN}(\xv^\ell)\in\R^{n\times w}$.
The latent features are first projected and rearranged into $H$ attention heads,
\begin{align}
\zv^\ell
=
\mathrm{rearrange}\left(
\yv^\ell\Pv_\ell^T
\right)
\in\R^{H\times n\times C},
\label{eq:transolver_projection}
\end{align}
where $\Pv_\ell\in\R^{HC\times w}$ is learnable and $C$ denotes the hidden dimension of each head.
In the following, the head index is suppressed for readability.

Within each head, the mesh points are assigned to $M_s$ physical states through softmax weights, i.e., 
\begin{align}
\boldsymbol{\alpha}^{\ell}
&=
\mathrm{softmax}
\left(
\zv^{\ell}\Wv_{w_\ell}^{T} / 
\tau^\ell
\right),
\qquad
\tau^\ell
=
\mathrm{MLP}(\zv^\ell),
\nonumber\\
\sv^{\ell}
&=
(\Dv^{\ell})^{-1}
(\boldsymbol{\alpha}^{\ell})^{T}
\zv^{\ell}\Wv_{p_\ell}^{T},
\label{eq:transolver_slice}
\end{align}
where $\Wv_{w_\ell}\in\R^{M_s\times C}$ and $\Wv_{p_\ell}\in\R^{C\times C}$ are learnable weights, and
$\Dv^\ell_{jj} = \sum_{i=1}^{n} \boldsymbol{\alpha}^\ell_{ij}$.
The resulting sliced tokens $\sv^\ell\in\R^{M_s\times C}$ represent the learned physical states.

Self-attention is then performed over these physical states as
\begin{align}
\widetilde{\sv}^{\ell}
=
\mathrm{softmax}
\left(
\frac{\Qv^{\ell}(\Kv^{\ell})^{T}}{\sqrt{C}}
\right)
\Vv^{\ell},
\label{eq:transolver_attention}
\end{align}
where
$\Qv^{\ell} = \sv^{\ell}\Wv_{Q_\ell}^{T}$, $\Kv^{\ell} = \sv^{\ell}\Wv_{K_\ell}^{T}$, $\Vv^{\ell} = \sv^{\ell}\Wv_{V_\ell}^{T}$, with $\Wv_{Q_\ell},\Wv_{K_\ell},\Wv_{V_\ell}\in\R^{C\times C}$ being learnable.
The attended physical states are mapped back to the mesh points using the same assignment weights, thus
\begin{align}
\pazocal{A}(\yv^\ell)
=
\mathrm{rearrange}\left(
\boldsymbol{\alpha}^{\ell}\,\widetilde{\sv}^{\ell}
\right)
\Wv_{r_\ell}^{T},
\label{eq:deslice}
\end{align}
where $\boldsymbol{\alpha}^{\ell}\widetilde{\sv}^{\ell}\in\R^{n\times C}$ is computed per head, the heads are concatenated by $\mathrm{rearrange}$ into $\R^{n\times HC}$, and $\Wv_{r_\ell}\in\R^{w\times HC}$ is again learnable.

The output $\xv^{n_L}$ of the final Transolver block defines the latent feature representation
$\boldsymbol{\Phi}:=\xv^{n_L}\in\R^{n\times w}$, from which a linear readout produces the search direction $\dv$.
The employed Transolver architecture is illustrated in Figure~\ref{fig:transolver}.

\begin{figure}[t]
\centering
\resizebox{\linewidth}{!}{%
\includegraphics{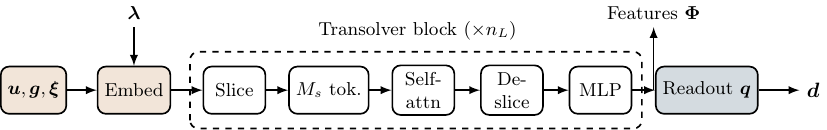}
}
\caption{Architecture of the Transolver used to generate the SNAP search direction.}
\label{fig:transolver}
\end{figure}

\subsection{Dataset construction and training}
\label{subsec:training_ons}
To train NOs, we first construct a set of reference solutions
$\{(\boldsymbol{\lambda}_j,\uv_j^\star)\}_{j=1}^{N_s}$
by sampling the parameter space $\Lambda$ and solving~\eqref{eq:dp_merit} with a Newton method, initialized with zero\footnote{For the Bratu--Gelfand and nonlinear Helmholtz problems, this initialization selects the stable solution branch.}.
For a given training iterate $\uv$ associated with $\boldsymbol{\lambda}_j$, we evaluate
${\gv=\nabla\Psi(\uv;\boldsymbol{\lambda}_j)}$
and determine the NO's parameters by solving the following minimization problem:
\begin{align}
\thetav^\ast
\in
\operatorname*{arg\,min}_{\thetav\in\R^{n_w}}
\omega\L_{\mathrm{obj}} + \L_{\mathrm{err}},
\end{align}
where $\omega \in \R$ and $\L_{\mathrm{obj}}$ and $\L_{\mathrm{err}}$ are given as
\begin{align}
\L_{\mathrm{obj}}&:=\frac{1}{N_s}
\sum_{j=1}^{N_s}
\mathbb{E}_{\uv\sim\pazocal{S}(\uv_j^\star;\boldsymbol{\lambda}_j)}
\left[
\Psi\!\left(
\uv
+
\mathrm{NO}_{\thetav}
(\uv,\gv;\boldsymbol{\xi},\boldsymbol{\lambda}_j);
\boldsymbol{\lambda}_j
\right)
-
\Psi(\uv_j^\star;\boldsymbol{\lambda}_j)
\right], \nonumber\\
\L_{\mathrm{err}}&:=\frac{1}{N_s}
\sum_{j=1}^{N_s}
\mathbb{E}_{\uv\sim\pazocal{S}(\uv_j^\star;\boldsymbol{\lambda}_j)}
\left( \uv+ \mathrm{NO}_{\thetav}
(\uv,\gv;\boldsymbol{\xi},\boldsymbol{\lambda}_j) - \uv_{j}^{\star} \right)^{2}. \nonumber
\label{eq:training_formula}
\end{align}
This training loss therefore corresponds to the optimality gap~\eqref{eq:dp_eps} and the  Euclidean solution match.
Note that since $\Psi(\uv_j^\star;\boldsymbol{\lambda}_j)$ is independent of $\thetav$, this term may be omitted during training without changing the minimizer.

Training with~\eqref{eq:training_formula} requires differentiating the discrete objective $\Psi$ at the corrected iterate with respect to the NO output.
For the finite element problems considered in this work (detailed in Section~\ref{sec:impl}), these derivatives are computed using the automatic differentiation capabilities of Firedrake~\cite{rathgeber2016firedrake}.
Importantly, this allows the NO to be trained with the objective loss and the solution error simultaneously rather than with the Euclidean solution error only.
As shown in Section~\ref{sec:impact_training}, the latter is equivalent to the defect only up to the local spectral bounds and may therefore be less suitable for ill-conditioned problems. 

\subsubsection{Sampling}
\label{sec:sampling}
The distribution $\pazocal{S}(\uv_j^\star;\boldsymbol{\lambda}_j)$ generates training states representative of different stages of the nonlinear solve.
In particular, given an initial guess $\uv_j^{(0)}$ and reference solution $\uv_j^\star$, 
each iterate is obtained by one of the following sampling strategies: 
\begin{enumerate}[topsep=0pt]
    \item \textbf{Perturbing the initial guess}: adding random perturbations bounded by a maximum magnitude of $\eta$,
    \item \textbf{Sampling between the initial guess and the reference solution}: taking a convex combination $\alpha u_{j}^{(0)} + (1-\alpha) u_{j}^{\star}$ with a random coefficient $\alpha \in [0, 1]$ between the perturbed initial guess and the reference solution,
    \item \textbf{Perturbing the reference solution}: adding random noise scaled by $\eta$ relative to the range ($\max - \min$) of the reference solution, or
    \item \textbf{Sampling pre-smoothed iterates generated by $M_1$}: applying the pre-smoothing operator $M_1$ to the perturbed initial state to generate smoothed iterates.
\end{enumerate}
The probabilities $\rho_1,\ldots,\rho_4$ ($\sum_{j=1}^4\rho_j=1$) associated with each strategy, perturbation distributions, and noise levels are detailed in Section~\ref{sec:impl}.

\subsubsection{Input normalization}
The magnitudes of the iterate and gradient can vary substantially across problem parameters and stages of the nonlinear solve, making appropriate input scaling important for robust training and search direction prediction.
We therefore normalize the iterate and gradient as
\begin{align}
\hat{\uv}
=
\frac{\uv}{a_{\uv}},
\qquad
\hat{\gv}
=
\frac{\gv}{a_{\gv}},
\qquad
s_{\uv}
=
\log_{10}a_{\uv},
\qquad
s_{\gv}
=
\log_{10}a_{\gv},
\label{eq:input_normalization}
\end{align}
where
$a_{\uv}
=
\max\{\|\uv\|_\infty,\nu\}$, and
$a_{\gv}
=
\max\{\|\gv\|_\infty,\nu\}$,
with a small constant $\nu>0$ that prevents division by zero.
The nodal inputs are therefore
$(\boldsymbol{\xi},\hat{\uv},\hat{\gv},s_{\uv},s_{\gv})$,
together with the parameter embedding of $\boldsymbol{\lambda}$.
The log-magnitudes $s_{\uv}$ and $s_{\gv}$ retain the scale information removed by normalization, allowing the NO to predict search directions of appropriate magnitude.
For brevity, $\mathrm{NO}_{\thetav}(\uv,\gv;\boldsymbol{\xi},\boldsymbol{\lambda})$ denotes evaluation with this normalization throughout the manuscript.

\subsubsection{Enforcing Dirichlet boundary conditions}
To ensure that the NO search direction preserves the prescribed Dirichlet data, the constrained dofs are eliminated from the search direction.
Let $\mv\in\{0,1\}^n$ denote the indicator of the unconstrained dofs and $\boldsymbol{\Phi}\in\R^{n\times w}$ the latent feature field produced by the NO.
We mask the latent features as $\widetilde{\boldsymbol{\Phi}} := \operatorname{diag}(\mv)\boldsymbol{\Phi}$,
and apply a bias-free linear readout with $\qv\in\R^{w}$, i.e., $\dv = \widetilde{\boldsymbol{\Phi}}\qv = \operatorname{diag}(\mv)\boldsymbol{\Phi}\qv$.
Thus, $\dv$ vanishes on the Dirichlet dofs, ensuring that the corrected iterate $\uv+\dv$ satisfies the prescribed boundary conditions.
The same mask is applied to the gradient input, while the iterate input is expected to satisfy the prescribed Dirichlet data.
\section{Benchmark Problems and Implementation Details}
\label{sec:impl}

\subsection{Benchmark Problems}
We consider six nonlinear benchmark problems, which we use to evaluate the performance our SNAP framework.
The minimization problems are understood in the local sense, i.e., $u^\star(\lambda)$ denotes a local minimizer of $\Psi(\cdot;\lambda)$.

Throughout, the spatial domain is the unit square $\Omega:=(0,1)^2$ with boundary $\partial\Omega$, and $\partial_n(\cdot)$ denotes the outward normal derivative on $\partial\Omega$.
Moreover, the symbol $\pazocal{U}([a,b])$ denotes the uniform distribution on $[a,b]$.

\subsubsection{Minimal Surface (MS)}
\label{sec:min_surf}

We begin with the parametric minimal surface problem~\cite{gilbarg2001elliptic},
\begin{equation}
\min_{u\in\pazocal{Z}_\lambda}
\Psi(u;\lambda)
:=
\int_{\Omega}
\sqrt{1+\|\nabla u\|^2}
\,dx,
\label{eq:min_surf}
\end{equation}
where
$\pazocal{Z}_\lambda
:=
\left\{
u\in H^1(\Omega)
\ \middle|\
u=g_\lambda
\text{ on }\partial\Omega
\right\}$.
The boundary conditions are prescribed via
$ g_\lambda(x)
=
\sin(\lambda\pi x_1)
\cos(\lambda\pi x_2)$,  where $x=(x_1,x_2)$.
The parameter $\lambda$ controls the oscillation frequency of the prescribed boundary data and hence the geometry of the resulting minimal surface.
We sample  $\lambda$ as $\lambda\sim\pazocal{U}([1,4])$.

\subsubsection{Bratu--Gelfand Problem (BG)}
\label{sec:bratu}
Next, we consider the Bratu--Gelfand problem~\cite{bratu1914equations,gelfand1963some}, defined by
\begin{equation}
\min_{u\in H_0^1(\Omega)}
\Psi(u;\lambda)
:=
\int_\Omega
\left(
\frac{1}{2}\|\nabla u\|^2
-
\lambda e^u
\right)
\,dx.
\label{eq:bratu_energy}
\end{equation}
Here, the parameter $\lambda>0$ controls the strength of the exponential nonlinearity.
For purpose of this work, we sample $\lambda\sim\pazocal{U}([0.5,3])$.
Note that $\Psi(\cdot;\lambda)$ is unbounded from below for every $\lambda>0$, since the exponential term dominates the gradient term for large $u$.
We therefore seek the stable, lower-branch solution of the associated Euler--Lagrange equation
\begin{equation}
-\Delta u=\lambda e^{u}\ \text{ in }\Omega,
\qquad
u=0\ \text{ on }\partial\Omega,
\label{eq:bratu}
\end{equation}
which is a strict local minimizer of $\Psi(\cdot;\lambda)$ for $\lambda$ below the fold value $\lambda\approx6.81$.

\subsubsection{Nonlinear Poisson (NP)}
\label{sec:nlpoisson}
Our next example is a nonlinear Poisson problem with solution-dependent conductivity.
Defining
$F_\lambda(u):=\lambda u+\frac{u^3}{3}$,
we consider the following minimization problem:
\begin{equation}
\min_{u\in\pazocal{Z}}
\Psi(u;\lambda)
:=
\int_\Omega
\left[
\frac{1}{2}
\left\|
\nabla F_\lambda(u)
\right\|^2
-
F_\lambda(u)
\right]
\,dx,
\label{eq:nlpoisson_energy}
\end{equation}
where $ \pazocal{Z}
:=
\left\{
u\in H^1(\Omega)
\ \middle|\
u=1
\text{ on }\Gamma_{\mathrm D}
\right\}$,
and 
$\Gamma_{\mathrm D}
:=
\left\{
x\in\partial\Omega
\ \middle|\
x_1=1
\right\}$.
The parameter $\lambda$ determines the baseline conductivity.
As $\lambda\to0$, the conductivity becomes increasingly dominated by the solution-dependent contribution $u^2$.
Here, we parametrize the problem by sampling $\lambda$ as $\lambda\sim\pazocal{U}([10^{-3},1])$.

\subsubsection{Nonlinear Helmholtz (NH)}
\label{sec:helmholtz}

As a next example, we utilize the cubically nonlinear Helmholtz problem,
\begin{equation}
\min_{u\in H_0^1(\Omega)}
\Psi(u;\lambda)
:=
\int_\Omega
\left(
\frac{1}{2}\|\nabla u\|^2
-
\frac{k^2}{2}u^2
-
\frac{\lambda}{4}u^4
-
u
\right)
\,dx.
\label{eq:helmholtz_energy}
\end{equation}
For this parametric problem, we fix the wavenumber at $k=4$ and sample ${\lambda\sim\pazocal{U}([0.1,5])}$, where $\lambda$ controls the strength of the cubic nonlinearity.
For $ \lambda> 0$, the energy $\Psi(\cdot;\lambda)$ is unbounded from below due to its negative quartic term.
Since $k^2=16<2\pi^2$, the quadratic part of $\Psi$ is coercive, ensuring a branch of strict local minimizers for sufficiently small $\lambda>0$.
We therefore seek solutions on this branch for the sampled parameter values, satisfying
\begin{equation}
\begin{aligned}
-\Delta u-k^{2}u-\lambda u^{3} &=1,
\quad \text{in }\Omega,\ \\
u&=0,
\quad \text{on }\partial\Omega.
\end{aligned}
\label{eq:helmholtz_el}
\end{equation}

\subsubsection{Allen--Cahn (AC)}
\label{sec:allencahn}

Next, we utilize the forced Allen--Cahn problem, given as
\begin{equation}
\min_{u\in H_0^1(\Omega)}
\Psi(u;\lambda)
:=
\int_\Omega
\left(
\frac{\varepsilon^2}{2}\|\nabla u\|^2
+
\frac{1}{4}u^4
-
\frac{1}{2}u^2
-
\lambda f u
\right)
\,dx.
\label{eq:allencahn_energy}
\end{equation}
We set
$\varepsilon=0.3$
and
$f(x)=\sin(\pi x_1)\sin(\pi x_2)$.
The bistable potential renders the energy nonconvex, while $\lambda$ controls the forcing amplitude.
We sample $\lambda\sim\pazocal{U}([0.1,5])$.

\subsubsection{Jumping-Coefficient \texorpdfstring{$p$}{p}-Laplacian (JC)}
\label{sec:plaplacian}

Finally, we employ a regularized $p$-Laplacian with a discontinuous diffusion coefficient, given as
\begin{equation}
\min_{u\in H_0^1(\Omega)}
\Psi(u;\lambda)
:=
\int_\Omega
\left[
\frac{K(x)}{p}
\left(
\varepsilon^2+\|\nabla u\|^2
\right)^{p/2}
-
\lambda f u
\right]
\,dx.
\label{eq:plaplacian_energy}
\end{equation}
We set
$p=4$,
$\varepsilon=0.1$,
and
$f(x)=\sin(\pi x_1)\sin(\pi x_2)$.
The coefficient $K$ is piecewise constant on a $2\times2$ checkerboard, in particular we choose $K$ as
\begin{equation}
K(x)
=
\begin{cases}
K_{\mathrm{high}},
&
\sin(2\pi x_1)\sin(2\pi x_2)\ge0,
\\[2pt]
K_{\mathrm{low}},
&
\text{otherwise},
\end{cases}
\label{eq:checkerboard}
\end{equation}
with
$K_{\mathrm{high}}=100$
and
$K_{\mathrm{low}}=1$.
The parameter $\lambda$ that controls the forcing amplitude is sampled as
$\lambda\sim\pazocal{U}([1,50])$.

\subsection{Implementation Details}
\label{subsec:implementation}
We use the Firedrake framework~\cite{rathgeber2016firedrake} for the finite element (FE) discretization of all benchmark problems, dataset generation, and automatic differentiation of the discrete objective loss~\eqref{eq:training_formula}.
All problems are approximated on quasi-uniform meshes using Lagrange $\mathbb{P}_1$ finite elements, with the parameter $\lambda$ sampled independently from the corresponding distributions specified above.
Reference solutions are computed using Newton's method with LS globalization.

The NOs are implemented in PyTorch~\cite{paszke2019pytorch} and trained using the SOAP optimizer~\cite{ICLR2025_e9886640}.
Training states are generated using the sampling strategy described in Section~\ref{subsec:training_ons} with $\eta=0.01, \rho_{1}=0.2, \rho_{2}=0.3, \rho_{3}=0.2$, and $\rho_{4}=0.3$.
The loss function~\eqref{eq:training_formula} is used with $\omega=0.01$ or $\omega=0$ (Solution-match only).
Moreover, we use a learning rate of $10^{-3}$, and a weight decay of $5\times10^{-4}$.
Training is terminated if the validation loss does not improve for $100$ consecutive epochs.
NO's architectures, associated dataset sizes, and training times are summarized in Appendix~\ref{sec:appB}.
Note, no extensive hyperparameter tuning is performed, suggesting that the reported results could be further improved through systematic hyperparameter optimization.

The SNAP algorithm\footnote{The developed code~\cite{precond_git_nonlinear} will be made publicly available upon acceptance of the manuscript.} is implemented using PETSc~\cite{balay2019petsc} through the petsc4py interface.
The numerical experiments are performed on the Oscar supercomputer at Brown University\footnote{Each computing node is equipped with an AMD EPYC 9554 64-Core Processor (256 GB) and an NVIDIA L40S GPU (48 GB).} and the Jean-Zay supercomputer at IDRIS\footnote{Each computing node is equipped with eight NVIDIA V100 GPUs (32 GB each) and a 24-core processor (360 GB).}.
\section{Numerical results}
\label{sec:num_results}

In this section, we evaluate the performance of the proposed SNAP framework on the six benchmark problems introduced in Section~\ref{sec:impl}. 
We consider SNAP equipped with three NOs, namely F-FNO, FNO, and Transolver (TS), whose configurations are summarized in Table~\ref{tab:architecture}. 
The performance of all methods is measured by the number of iterations required to reach the prescribed stopping tolerance; a solver is marked as FAILED if it does not converge within $20{,}000$ iterations.
For each problem, we uniformly sample five values of $\lambda$ from the interval specified in Section~\ref{sec:impl}, and report the mean and standard deviation (std) of the iteration counts over five runs.

\subsection{Impact of the training loss on SNAP convergence}
\label{sec:effect_loss}

\begin{table}
   \centering
   \begin{tabular}{c||c|c}
   \toprule
    Ex./Loss choice &  $\omega\L_{\mathrm{obj}} + \L_{\mathrm{err}}$ &$ \L_{\mathrm{err}}$ \\ \midrule\midrule
    MS & $756.0 \pm 107.5$ & $\mathbf{724.8} \pm 157.1$  \\ \midrule
    BG & $110.4 \pm 42.3$ & $\mathbf{107.2} \pm 22.4$   \\ \midrule
    NH & $\mathbf{1{,}358.6} \pm 278.5$ & $3{,}844.0 \pm 2018.9$  \\ \midrule
    NP & $\mathbf{338.6} \pm 26.9$ & $4{,}268.8 \pm 7865.6$ \\ \midrule
    AC & $\mathbf{356.8} \pm 293.7$ & $1{,}325.4 \pm 2340.2$  \\ \midrule
    JC & $\mathbf{219.3} \pm 41.2$ & $281.6 \pm 9.2$ \\ \midrule
   \bottomrule
   \end{tabular}
   \caption{Mean and std of the number of iterations required for SNAP with F-FNO to converge, as a function of the NO training loss.}
   \label{tab:ablation_loss}
\end{table}

We first investigate the impact of the training loss on SNAP convergence.
Proposition~\ref{cor:dp_training} motivates objective-based training, as it directly targets the correction defect governing the reduction achieved by an accepted unit step, whereas the Euclidean solution-match loss controls this defect only up to the spectral constants $\underline{\lambda}$ and $\overline{\lambda}$.
We compare SNAP(F-FNO) trained with the combined energy and solution-match loss against solution-match loss alone ($\omega=0$), reporting iteration counts in Table~\ref{tab:ablation_loss}.

As we can see, including the energy term clearly improves the performance on four of the six problems, whereas the solution-match loss alone performs slightly better on MS and BG. 
For MS ($756.0$ vs.\ $724.8$) and BG ($110.4$ vs.\ $107.2$), however, the differences are well within one std and therefore not conclusive.
The differences are large for NP, AC, and NH. 
The energy-based training reduces the iteration count from $4{,}268.8$ to $338.6$ on NP ($\approx 12.6\times$), from $1{,}325.4$ to $356.8$ on AC ($\approx 3.7\times$), and from $3{,}844.0$ to $1{,}358.6$ on NH ($\approx 2.8\times$). 
For NP and AC, the std of the solution-match variant exceeds its mean, which indicates that the solution-match-trained NO provides a poor search direction at particular values of $\lambda$, leading to very slow convergence for those parameters, whereas the energy-based training yields a much more consistent behavior across the parameter range $\Lambda$.
The improvement is smaller but systematic for JC ($\approx 1.3\times$), where the difference exceeds the stds of both variants. 
Overall, the energy-based loss improves performance on most problems without substantially degrading it on the others.
We therefore adopt it for all subsequent experiments.

\subsection{Convergence of SNAP as a standalone solver}
\label{sec:effect_no}
Next, we investigate the benefit of the NO's search direction when SNAP is used as a standalone solver. 
As baselines, we use standard gradient descent~(GD) and the nonlinear Jacobi preconditioned gradient descent~(NJ-GD), in which the gradient is scaled by the diagonal of the Hessian $\nabla^2\Psi$.  
The results are summarized in Table~\ref{tab:ablation_neural_operator}.

Across all benchmark problems, SNAP requires substantially fewer iterations than both baselines.
For the MS, BG, and AC problems, the reduction in the mean iteration count relative to GD ranges from approximately $2\times$ to $27\times$ across the SNAP variants. 
For instance, SNAP(F-FNO) reduces the number of iterations from $6{,}281.8$ to $756.0$ on MS ($\approx 8.3\times$) and from $2{,}972.2$ to $110.4$ on BG ($\approx 27\times$), while SNAP(FNO) reduces it from $1{,}202.6$ to $220.8$ on AC ($\approx 5.4\times$). 
An even larger reduction is observed for the NH problem, where SNAP(TS) and SNAP(F-FNO) require $177.2$ and $1{,}358.6$ iterations, respectively, compared to $15{,}638.8$ for GD ($\approx 88\times$ and $\approx 12\times$). 
Moreover, the NJ-GD provides only mild improvement over GD on BG, NH, and AC (approximately $11-14\%$ fewer iterations), indicating that diagonal scaling alone is insufficient to capture the global structure of the error.

SNAP's benefits are particularly pronounced on the most challenging problems.
 For example, for NP, GD fails to converge, while SNAP reduces the iteration count by $8.3-9.6\times$ relative to NJ-GD.
 For JC, both GD and NJ-GD fail to converge within the iteration budget, whereas all SNAP variants converge in approximately $219-356$ iterations.
 These results demonstrate that learned corrections can enable convergence where the underlying optimization methods alone are ineffective.
 
We also observe that the effectiveness of SNAP depends on the NO architecture, and no single architecture is uniformly best.
F-FNO yields the lowest iteration counts on MS, BG, and JC, FNO on NP and AC, and TS on NH. 
The architecture dependence is most pronounced on NH, where the iteration counts range from $177.2$ for SNAP(TS) to $2{,}395.4$ for SNAP(FNO). 
We further note that the stds are sizable for several configurations (e.g., SNAP(F-FNO) and SNAP(TS) on AC, and SNAP(FNO) on JC), which indicates that the quality of the learned search direction varies across the parameter range $\Lambda$.
 
Figure~\ref{fig:convergence_allen_cahn_lambda} shows the convergence histories of all considered methods for the AC problem with $\lambda=4.09$.
The SNAP variants reduce both the gradient norm and the objective function substantially faster than GD and NJ-GD.
The observed objective function decreases monotonically, consistent with Lemma~\ref{lem:dp_global}.

\begin{table}
   \centering
      \resizebox{\linewidth}{!}{
   \begin{tabular}{c||c|c|c|c|c}
   \toprule
     \multirow{2}{*}{  Ex.} &\multirow{2}{*}{ GD}  &  \multirow{2}{*}{ NJ-GD}  & \multicolumn{3}{c}{SNAP (NJ-GD, NO)} \\ \cline{4-6}
     &  & & F-FNO & FNO & TS \\ \midrule\midrule
    MS & $6{,}281.8 \pm 912.9$  & $2{,}970.4 \pm 1{,}855.8$ & $\mathbf{756.0} \pm 107.5$ & $ 794.8 \pm 155.0$ & $908.0 \pm 137.3$ \\ \midrule
    BG & $2{,}972.2 \pm 150.5$ & $2{,}577.9 \pm 128.5$ & $\mathbf{110.4} \pm 42.3$ & $202.0 \pm 110.0$ & $431.4 \pm 107.9$  \\ \midrule
    NH & $15{,}638.8 \pm 2{,}818.8$ & $13{,}473.8 \pm 2{,}682.9$ & $1{,}358.6 \pm 278.5$ & $2{,}395.4 \pm 530.4$ & $\mathbf{177.2} \pm 60.3$  \\ \midrule
    NP & FAILED &  $2{,}799.2 \pm 249.9$ & $338.6 \pm 26.9$ & $\mathbf{291.4} \pm 21.8$ & $294.0 \pm 34.9$ \\ \midrule
    AC & $1{,}202.6 \pm 456.7$ & $1{,}068.6 \pm 707.2$ & $356.8 \pm 293.7$ & $\mathbf{220.8} \pm 112.7$ & $557.2 \pm 458.2$ \\ \midrule
    JC & FAILED & FAILED & $\mathbf{219.3} \pm 41.2$ & $356.2 \pm 217.7$ & $261.8 \pm 34.9$ \\ \midrule
   \bottomrule
   \end{tabular}}
   \caption{The mean and std of the required number of iterations until convergence of the standard gradient descent (GD), the nonlinear Jacobi (the gradient is scaled with the diagonal of the Hessian)  preconditioned gradient descent~(NJ-GD), and SNAP methodology with NJ-GD hybridized by F-FNO, FNO, and TS, denoted as SNAP(F-FNO), SNAP(FNO), and SNAP(TS), respectively.}
   \label{tab:ablation_neural_operator}
\end{table}

\begin{figure}
    \centering
\includegraphics{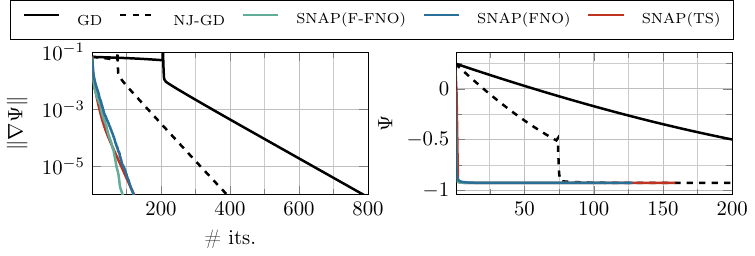}
	\caption{Convergence of GD variants solving AC example with $\lambda=4.09$.}
    \label{fig:convergence_allen_cahn_lambda}
\end{figure}

\subsection{SNAP as a nonlinear preconditioner}
\label{sec:precond}
In the next experiment, we use SNAP as a nonlinear preconditioner for nonlinear GMRES~(N-GMRES) with history size $30$. 
We compare unpreconditioned N-GMRES, N-GMRES preconditioned by the NJ-GD, and SNAP, which hybridizes NJ-GD with the three considered NOs. 
Table~\ref{tab:comparison_ngmres} reports the iteration counts.
 
Compared to the NJ-GD preconditioner, the SNAP preconditioner reduces the number of iterations on four of the six problems. 
For MS, the reduction is $1.9$--$2.0\times$ (from $162.0$ to $79.8$--$86.4$); for NP, it is $1.7$--$2.1\times$ (from $213.6$ to $99.4$--$129.0$); for AC, it is $1.7$--$1.9\times$ (from $106.2$ to $56.6$--$63.6$); and for JC, it is $1.4$--$1.7\times$ (from $141.4$ to $82.4$--$99.8$). 
Moreover, unpreconditioned N-GMRES fails on JC, whereas all SNAP-preconditioned variants converge in approximately $100$ iterations; see also the convergence histories in Figure~\ref{fig:convergence_jumping_p_laplacian_lambda}.
The ranking of the NO architectures also differs from the standalone setting of Table~\ref{tab:ablation_neural_operator}. 
There, TS was the best-performing architecture only on NH, whereas here SNAP(TS) yields the lowest iteration counts for BG, NP, AC, and JC problems. 
This indicates that the ranking of the NO architectures in the standalone setting does not necessarily carry over to the preconditioned setting.

For BG and NH, SNAP does not clearly improve performance. 
On BG, N-GMRES, and the preconditioned variant with NJ-GD, SNAP(F-FNO), and SNAP(FNO) preconditioners all converge in approximately $53$--$56$ iterations, so there is little room for improvement. 
On NH, N-GMRES and its NJ-GD preconditioned counterpart already converge in $56$--$58$ iterations, while the SNAP-preconditioned variants require more iterations. 
NH is therefore the case where the additional cost of the NO-based preconditioner does not pay off in terms of iteration counts. 
We also note that the iteration counts reported here do not account for the per-iteration cost, which is higher for SNAP than for the NJ-GD preconditioner, since each application of the preconditioner involves an NO evaluation in addition to the NJ-GD smoothing.

\begin{table}
   \centering
   \resizebox{\linewidth}{!}{
   \begin{tabular}{l||c|c|c|c|c}
   \toprule 
     \multirow{2}{*}{  Ex.} &\multirow{1}{*}{N-GMRES}  &  \multirow{1}{*}{N-GMRES}  & \multicolumn{3}{c}{N-GMRES (SNAP (NJ-GD, NO))} \\ \cline{4-6}   
     & (None) & (NJ-GD) & F-FNO & FNO & TS \\ \midrule\midrule
    MS  & $191.2 \pm 28.8$ & $162.0 \pm 4.5$ & $81.8 \pm 5.4$ & $\mathbf{79.8} \pm 7.4$ & $86.4 \pm 5.8$ \\ \midrule
    BG  & $55.6 \pm 0.5$ & $55.4 \pm 0.5$ & $54.0 \pm 2.8$ & $53.4 \pm 1.5$ & $\mathbf{45.2} \pm 5.9$  \\ \midrule
    NH  & $58.0 \pm 0.0$ & $\mathbf{56.2} \pm 0.4$ & $70.0 \pm 3.8$ & $64.2 \pm 3.1$ & $75.6 \pm 0.8$  \\ \midrule
    NP  & $333.4 \pm 109.1$ & $213.6 \pm 48.1$ & $129.0 \pm 21.4$ & $102.0 \pm 10.4$ & $\mathbf{99.4} \pm 6.6$ \\ \midrule
    AC & $226.8 \pm 147.8$ & $106.2 \pm 49.1$ & $63.6 \pm 13.9$ & $61.0 \pm 9.5$ & $\mathbf{56.6} \pm 10.1$ \\ \midrule
    JC & FAILED & $141.4 \pm 17.0$ & $88.0 \pm 8.9$ & $99.8 \pm 13.3$ & $\mathbf{82.4} \pm 4.4$ \\ \midrule
   \bottomrule
   \end{tabular}
   }
   \caption{The mean and std of the required number of iterations until convergence of the N-GMRES with the history size $30$, the nonlinear Jacobi preconditioned N-GMRES, and the preconditioned N-GMRES using SNAP methodology with F-FNO, FNO, and TS, denoted as SNAP(F-FNO), SNAP(FNO), and SNAP(TS), respectively.}
   \label{tab:comparison_ngmres}
\end{table}

\begin{figure}
    \centering
\includegraphics{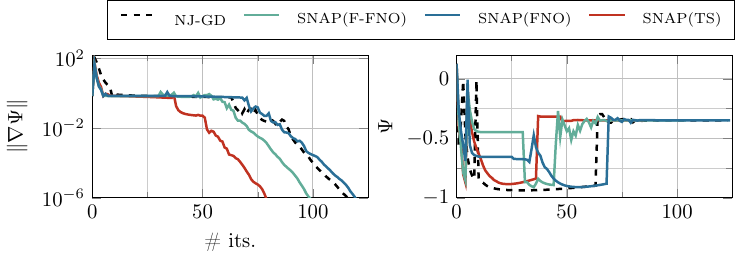}
    \caption{Convergence of N-GMRES with NJ-GD and SNAP (NJ-GD, NO) preconditioners for JC example with $\lambda=40.85$.}
    \label{fig:convergence_jumping_p_laplacian_lambda}
\end{figure}

\subsection{Scalability of SNAP with increasing resolution}
\label{sec:scalability}
 
Finally, we investigate SNAP's performance as a nonlinear preconditioner under mesh refinement for the NP problem.
 We vary the number of unknowns per axis, $N_x$, and consider three parameter subintervals, 
 $\lambda\in[10^{-3},10^{-2}]$, $[10^{-2},10^{-1}]$, and $[10^{-1},1]$, 
 corresponding to progressively weaker solution-dependent nonlinearity.
 We compare N-GMRES without preconditioning against N-GMRES preconditioned by NJ-GD, SNAP(F-FNO), and SNAP(TS).
 Figure~\ref{fig:scalability} presents the results.
 
The number of iterations for N-GMRES and NJ-GD preconditioned variant increases with mesh resolution, especially for small $\lambda$, where the problem is more nonlinear and ill-conditioned. 
In contrast, the SNAP-preconditioned methods require far fewer iterations and remain below the baselines across the considered resolutions. 
This behavior holds across all three subintervals of $\lambda$ and indicates that the learned search direction generalizes across mesh resolutions.

\begin{figure}
    \centering
\includegraphics{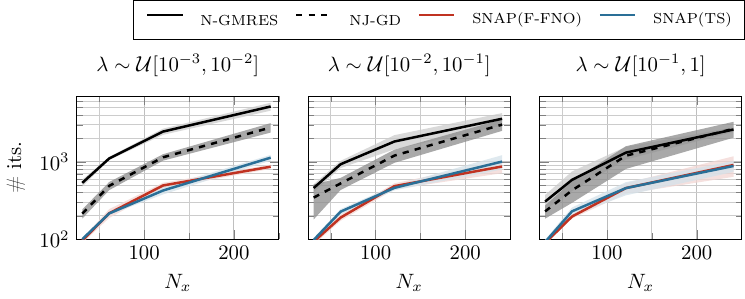}
	\caption{The average number of iterations required to solve the NP example as the number of dofs increases ($n=(N_X)^2$). We uniformly sample five $\lambda$ values from the three subintervals $[10^{-3}, 10^{-2}]$, $[10^{-2}, 10^{-1}]$, and $[10^{-1}, 1]$ and report the average number of iterations and the corresponding confidence interval. Note that NJ-GD, SNAP(F-FNO), and SNAP(TS) are used as preconditioners for N-GMRES.}
\label{fig:scalability}
\end{figure}
\section{Conclusion}
 \label{sec:conclusion}
We introduced the Safeguarded Neurally-Accelerated Preconditioning (SNAP) framework, which integrates pre-trained neural operators (NOs) into classical nonlinear optimization methods to accelerate the solution of parametric minimization problems.
An a posteriori safeguard ensures global convergence and preserves the underlying optimization method's convergence rate, regardless of the NO architecture.
Our analysis further quantifies the influence of learned correction accuracy on convergence acceleration and establishes a direct connection to the NO training loss, providing a theoretical motivation for objective-based training.

We evaluated SNAP with F-FNO, FNO, and Transolver on six nonlinear parametric problems.
As a standalone solver, SNAP reduced gradient descent iteration counts by factors of approximately $2-88\times$ and converged in cases where the baseline failed.
As a nonlinear preconditioner for N-GMRES, SNAP reduced iteration counts by up to a factor of two on four benchmark problems, while achieving comparable performance on the remaining two.
We further demonstrated that these improvements persist under mesh refinement and generalize across problem parameters.

Although SNAP provides convergence guarantees, the acceleration achieved in practice depends on the problem and NO architecture.
Moreover, finite NO prediction accuracy may limit the effectiveness of learned corrections near convergence, increasing reliance on the underlying optimization method.
Our numerical evaluation is also restricted to two-dimensional problems and focuses on iteration counts rather than wall-clock performance.
Future work will investigate active learning and adaptive, solver-aware training strategies to improve the accuracy and effectiveness of learned corrections.
We also plan to explore hybridization with more advanced optimization methods, including Newton-type methods, extend SNAP to three-dimensional problems, and systematically assess its computational efficiency and scalability.
 \end{sloppypar}

\appendix
\section{Numerical Approximation Details}
\label{sec:appB}
This section summarizes the NO architectures (\Cref{tab:architecture}) and reports the dataset sizes, training epochs, and training times in~\Cref{table:training_params}.
 Here, an epoch consists of a single randomly sampled training instance rather than a full pass over the $N_S=2{,}000$ samples.
 Each epoch therefore requires only one forward and backward pass, followed by a convergence check.
 Consequently, training takes only a few seconds to approximately one minute, as reported in~\Cref{table:training_params}.
 
\begin{table}
\centering
\begin{tabular}{r||r|r|r|r|r|r}
\toprule
    &  \textbf{Layers} & \textbf{Modes} & \textbf{Widths} & \textbf{FFT} & \textbf{Act.} &\textbf{Params} \\ \midrule \midrule
{FNO}		& 3 & 1	& 128 & 2D & GELU & 247,425 \\ \midrule
{F-FNO}		& 3 & 1 & 128 & 1D & GELU & 149,121\\
\midrule\midrule
   &  \textbf{Layers} & \textbf{Heads} & \textbf{Channels} & \textbf{Slices} & \textbf{Act.} &\textbf{Params} \\ \midrule
{Transolver} 	& 3 & 4 & 128 & 32 & GELU & 296,940	 \\
\bottomrule
\end{tabular}
\caption{The summary of NO's architectures.}
\label{tab:architecture}
\end{table}

\begin{table}
\centering
\begin{tabular}{r||r|r|r|r}
\toprule
\multirow{2}{*}{}   		& \multirow{2}{*}{} 	& \multicolumn{3}{c}{\textbf{\textbf{Epoch / Time  (secs)}}}	\\ \cmidrule{3-5} \cmidrule{3-5}
   \textbf{Example} 		&  $\mathbf{N_S/Model}$	& \textbf{FNO} & \textbf{F-FNO} & \textbf{Transolver}  	\\ \midrule \midrule
MS				& 2,000				& 138 / 3.80 & 138 / 2.69 & 205 / 5.63			   	\\ \midrule
BG 			& 2,000				& $1{,}451$ / 38.76 & $1{,}262$ / 33.63 & $1{,}671$ / 30.55 		   			\\
\midrule
NH 	& 2,000	         & $1{,}298$ / 34.63 & $1{,}138$ / 21.91 & $1{,}451$ / 38.76	 		\\
\midrule
NP 				& 2,000				& 1,254 / 25.08 & 221 / 3.04 & 301 / 6.99 					 \\ \midrule
AC 				& 2,000				& $1{,}971$ / 53.96 & $1{,}959$ / 38.13 & $1{,}999$ / 54.99 					 \\
\midrule
JC				& 2,000				& 466 / 13.39 & 308 / 6.91 & 145 / 4.29 \\
\bottomrule
\end{tabular}
\caption{Summary of the number of samples ($N_S$), epoch, and training time for all benchmark problems.}
\label{table:training_params}
\end{table}

\bibliographystyle{siamplain}
\bibliography{biblio}

\end{document}